\documentclass{article}
\usepackage{float}
\usepackage{amsfonts}
\usepackage{pst-node,pst-circ}
\usepackage{amsthm}
\usepackage{amssymb,color}
\usepackage{mdframed}
\usepackage[leqno]{amsmath}

\theoremstyle{plain}
\newtheorem{theorem}{Theorem}[section]
\newtheorem{lemma}[theorem]{Lemma}
\newtheorem{proposition}[theorem]{Proposition}

\theoremstyle{definition}
\newtheorem{definition}[theorem]{Definition}
\newtheorem{remark}[theorem]{Remark}
\newtheorem{example}[theorem]{Example}
\usepackage[top=3cm,bottom=3cm,left=3.2cm,right=3.2cm,headsep=10pt]{geometry}

\usepackage[shortlabels]{enumitem}
\usepackage{hyperref}
\usepackage{nameref}
\usepackage{amsmath}

\makeatletter

\usepackage{todonotes}

\newcommand{\proj}{\operatorname{proj}}
\newcommand{\Proj}{\operatorname{Proj}}
\renewcommand{\H}{\mathcal{H}}  
\newcommand{\B}{\mathbb{B}}
\newcommand{\R}{\mathbb{R}}
\newcommand{\N}{\mathbb{N}}
\def\tto{\rightrightarrows}

\title{Inexact Catching-Up Algorithm for Moreau's Sweeping Processes}
\author{Juan Guillermo Garrido\footnote{Departamento de Ingeniería Matemática, Universidad de Chile, Santiago, Chile. E-mail: \href{mailto:jgarrido@dim.uchile.cl}{jgarrido@dim.uchile.cl}},  Maximiliano Lioi\footnote{Departamento de Ingeniería Matemática, Universidad de Chile, Santiago, Chile. E-mail: \href{mlioi@dim.uchile.cl}{mlioi@dim.uchile.cl}}  \,and Emilio Vilches\footnote{Instituto de Ciencias de la Ingeniería, Universidad de O’Higgins, Rancagua, Chile.
		E-mail: \href{mailto:emilio.vilches@uoh.cl}{emilio.vilches@uoh.cl}}}
\date{\today}

\begin{document}

\maketitle

\begin{abstract}
In this paper, we develop an inexact version of the catching-up algorithm for sweeping processes. We define a new notion of approximate projection, which is compatible with any numerical method for approximating exact projections, as this new notion is not restricted to remain strictly within the set. We provide several properties of the new approximate projections, which enable us to prove the convergence of the inexact catching-up algorithm in three general frameworks: prox-regular moving sets, subsmooth moving sets, and merely closed sets. Additionally, we apply our numerical results to address complementarity dynamical systems, particularly electrical circuits with ideal diodes. In this context, we implement the inexact catching-up algorithm using a primal-dual optimization method, which typically does not necessarily guarantee a feasible point. Our results are illustrated through an electrical circuit with ideal diodes. Our results recover classical existence results in the literature and provide new insights into the numerical simulation of sweeping processes.
\end{abstract}

\section{Introduction}

The sweeping process, originally introduced by J.-J. Moreau in a series of foundational papers \cite{MR637727, MR637728}, was motivated by various concrete applications, including quasi-static evolution in elastoplasticity, contact dynamics, and friction dynamics \cite{Moreau_2004, Moreau_2011}. Since then, it has garnered significant interest in the study of dynamical systems with time-dependent constraints, particularly in fields such as nonsmooth mechanics, crowd motion \cite{MR3467591,MR2404902}, and, more recently, the modeling of electrical circuits \cite{MR2732595}. A well-established existence theory for the sweeping process is now widely recognized (see, e.g., \cite{MR2159846, MR3286703}). Of particular  interest is the case of prox-regular moving sets, where  existence and uniqueness of solutions can be established using the so-called catching-up algorithm (see \cite{MR2159846}). Originally introduced by J.-J. Moreau in \cite{MR637728} for convex sets, the catching-up algorithm consist of an implicit discretization of the sweeping process, resulting in an iteration based on the projection onto the moving set.   

The numerical applicability of the catching-up algorithm relies on the ability to compute an exact projection formula for the moving sets. However, for most sets, obtaining an exact projection onto a closed set is not feasible, and only numerical approximations can be obtained. In this paper, we develop a theoretical framework for numerically approximating solutions to the sweeping process. Building on the work done in \cite{MR2731287}, we define a new notion of approximate projection that is compatible with any numerical method for approximating exact projections, as this new notion is not restricted to remain strictly within the set.   

In this work, we provide several properties of approximate projections and propose a general numerical method for the sweeping process based on these approximate projections. This algorithm can be considered as an inexact version of the catching-up algorithm, for which we prove the convergence in three general cases: (i) prox-regular moving sets (without compactness assumptions), (ii) ball-compact subsmooth moving sets, and (iii) general ball-compact fixed closed set. As a result, our findings cover a wide range of existence for the sweeping process. 

Additionally, we apply our numerical results to address complementarity dynamical systems, particularly electrical circuits with ideal diodes. In this context, we implement the inexact catching- up algorithm using a primal-dual optimization method, which typically does not guarantee a feasible point.

It is worth emphasizing that our results generalize the catching-up algorithm and provide significant insights into the numerical solution of sweeping processes.

The paper is organized as follows. Section \ref{Sec-2} introduces the mathematical tools required for the presentation and develops the theoretical properties of approximate projections. Section \ref{Sec-3} focuses on the proposed algorithm and its main properties. In Section \ref{Sec-4}, we prove the convergence of the algorithm when the moving set has uniformly prox-regular values, without assuming compactness. Section \ref{Sec-5} addresses the convergence of the proposed algorithm for ball-compact subsmooth moving sets. Section \ref{Sec-6} extends this analysis to the case of a fixed ball-compact set. Finally, Section \ref{Sec-7} explores numerical techniques for tackling complementarity dynamical systems. In particular, we reformulate a specific class of these systems as a perturbed sweeping process, enabling the use of approximate projections. The paper ends with some concluding remarks.

\section{Mathematical Preliminaries}\label{Sec-2}
From now on, $\H$ denotes a real Hilbert space, whose norm, denoted by $\Vert \cdot \Vert$, is induced by the inner product $\langle \cdot, \cdot \rangle$. The closed (resp. open) ball centered at $x$ with radius $r > 0$ is denoted by $\B[x,r]$ (resp. $\B(x,r)$), and the closed unit ball is denoted by $\B$. For a given set $S \subset \H$, the \textit{support} and the \textit{distance} function of $S$ at $x \in \H$ are defined, respectively, as
$$
\sigma(x,S) = \sup_{z \in S} \langle x,z \rangle \text{ and }d_{S}(x) := \inf_{z \in S} \Vert x-z \Vert .
$$
A set $S\subset \H$ is called \emph{ball compact} if the set $S\cap r\mathbb{B}$ is compact for all $r>0$.
Given $\rho \in ]0,+ \infty]$ and $\gamma \in ]0, 1[$, the $\rho-$enlargement and the $\gamma \rho-$enlargement of $S$ are defined, respectively, as
$$
U_{\rho}(S) = \{x \in \H: d_S(x) < \rho\} \text{ and } U_{\rho}^{\gamma}(S) = \{x \in \H: d_S(x) < \gamma \rho\} .
$$
The Hausdorff distance between two sets $A, B \subset \mathcal{H}$ is defined as
$$
\operatorname{Haus}(A,B):=\max\{\sup_{x\in A}d_B(x),\sup_{x\in B}d_A(x)\}.
$$
A vector $h \in \mathcal{H}$ belongs to the Clarke tangent cone $T(S;x)$ (see \cite{MR1058436}); when for every sequence $\left(x_n\right)$ in $S$ converging to $x$ and every sequence of positive numbers $
(t_n)$ converging to 0, there exists a sequence $\left(h_n\right)$ in $\H$ converging to $h$ such that $x_n+t_n h_n \in S$ for all $n \in \mathbb{N}$. This cone is closed and convex, and its negative polar $N(S ; x)$ is the Clarke normal cone to $S$ at $x \in S$, that is,
$$
N(S ; x):=\{v \in \mathcal{H}:\langle v, h\rangle \leq 0 \text { for all } h \in T(S ; x)\} .
$$
As usual, $N(S ; x)=\emptyset$ if $x \notin S$. Through that normal cone, the Clarke subdifferential of a function $f\colon \mathcal{H} \rightarrow \mathbb{R} \cup\{+\infty\}$ is defined as
$$
\partial f(x):=\{v \in \mathcal{H}:(v,-1) \in N \text { (epi } f,(x, f(x)))\} ,
$$
where epi $f:=\{(y, r) \in \mathcal{H} \times \mathbb{R}: f(y) \leq r\}$ is the epigraph of $f$. When the function $f$ is finite and locally Lipschitzian around $x$, the Clarke subdifferential is characterized (see, e.g.,  \cite{MR1488695}) in the following simple and amenable way
$$
\partial f(x)=\left\{v \in \mathcal{H}:\langle v, h\rangle \leq f^{\circ}(x ; h) \text { for all } h \in \H\right\} ,
$$
where
$$
f^{\circ}(x ; h):=\limsup _{(t, y) \rightarrow\left(0^{+}, x\right)} t^{-1}[f(y+th)-f(y)],
$$
is the \emph{generalized directional derivative} of the locally Lipschitzian function $f$ at $x$ in the direction $h \in \mathcal{H}$. The function $f^{\circ}(x ; \cdot)$ is in fact the support of $\partial f(x)$, i.e., $f^{\circ}(x ; h)=\sigma_{\partial f(x)}(h)$. That characterization easily yields that the Clarke subdifferential of any locally Lipschitzian function is a set-valued map with nonempty and convex values satisfying the important property of upper semicontinuity from $\H$ into $\H_w$.
Let $f\colon \H \rightarrow \R \cup\{+\infty\}$ be an lsc (lower semicontinuous) function and $x \in \operatorname{dom} f$. We say that
\begin{enumerate}[(i)]
\item An element $\zeta \in \mathcal{H}$ belongs to the \emph{proximal subdifferential} of $f$ at $x$, denoted by $\partial_P f(x)$, if there exist two non-negative numbers $\sigma$ and $\eta$ such that
$$
f(y) \geq f(x)+\langle\zeta, y-x\rangle-\sigma\|y-x\|^2 \text { for all } y \in \mathbb{B}(x ; \eta) .
$$
\item An element $\zeta \in \mathcal{H}$ belongs to the \emph{Fr\'echet subdifferential} of $f$ at $x$, denoted by $\partial_F f(x)$, if
$$
\liminf _{h \rightarrow 0} \frac{f(x+h)-f(x)-\langle\zeta, h\rangle}{\|h\|} \geq 0 .
$$
\item An element $\zeta \in \mathcal{H}$ belongs to the \emph{limiting subdifferential} of $f$ at $x$, denoted by $\partial_L f(x)$, if there exist sequences $\left(\zeta_n\right)$ and $\left(x_n\right)$ such that $\zeta_n \in \partial_P f\left(x_n\right)$ for all $n \in \mathbb{N}$ and $x_n \rightarrow x, \zeta_n \rightharpoonup \zeta$, and $f\left(x_n\right) \rightarrow f(x)$.
\end{enumerate}
Through these concepts, we can define the proximal, Fr\'echet, and limiting normal cone of a given set $S \subset \mathcal{H}$ at $x \in S$, respectively, as
$$
N^P(S ; x):=\partial_P I_S(x), N^F(C ; x):=\partial_F I_C(x) \text { and } N^L(S ; x):=\partial_L I_S(x),
$$
where $I_S(x)=0$ if $x \in S$ and $I_S(x)=+\infty$ if $x \notin S$. It is well known that the normal cone and the distance function are related by the following formulas (see, e.g., \cite[Theorem 4.1]{MR1870754} and \cite{MR1488695}):
$$N^P(S ; x) \cap \mathbb{B}=\partial_P d_S(x) \textrm{ and } N(S ; x)=\overline{\operatorname{co}}^* N^L(S ; x)=\operatorname{cl}^*\left(\mathbb{R}_{+} \partial d_S(x)\right) \textrm{ for all } x\in S.
$$
\noindent In this paper, we consider the prominent class of prox-regular sets. Introduced by Federer in \cite{MR110078} and later developed by Rockafellar, Poliquin, and Thibault in \cite{MR1694378}.  The prox-regularity generalizes and unifies convexity and nonconvex bodies with $C^2$ boundary. We refer to \cite{MR2768810,MR4659163} for a survey.
\begin{definition}
Let $S$ be a closed subset of $\H$ and $\rho \in ]0, +\infty]$. The set $S$ is called $\rho-$uniformly prox-regular if for all $x \in S$ and $\zeta \in N^{P}(S;x)$ one has
$$\langle \zeta, x' - x \rangle \leq \frac{\Vert \zeta \Vert}{2 \rho} \Vert x' - x \Vert^{2} \text{ for all } x' \in S .$$
\end{definition}
\noindent It is clear that convex sets are $\rho$-uniformly prox-regular for any $\rho>0$. The following proposition provides a characterization of uniformly prox-regular sets (see, e.g., \cite{MR2768810}). 
\begin{proposition}\label{prop-prox-regular} 
Let $S \subset \H$ closed and $\rho \in ]0, +\infty]$. The following assertions are equivalent:
\begin{enumerate}[(a)]
    \item S is $\rho$-uniformly prox-regular.
    \item For any $\gamma\in ]0,1[$ the projection is well-defined on $U_{\rho}^{\gamma}(S)$ and for all $u_1, u_2 \in U_{\rho}^{\gamma}(S)$, one has
    $$\Vert \proj_{S}(u_1) - \proj_{S}(u_2) \Vert \leq (1 - \gamma)^{-1} \Vert u_1 - u_2 \Vert.$$
    \item For all $x_i \in S, v_i \in N^{P}(S;x_i)\cap \mathbb{B}$, with $i = 1, 2$, one has
    $$\langle v_1 - v_2, x_1 - x_2 \rangle \geq - \frac{1}{ \rho} \Vert x_1 - x_2 \Vert^{2},$$
    that is, the set-valued mapping $N^{P}(S;\cdot) \cap \B$ is $1/\rho-$hypomonotone.
    \item For all $\gamma \in ]0,1[$, for all $x', x \in U_{\rho}^{\gamma}(S)$ and for all $v \in \partial_{P} d_{S}(x)$, one has
    $$\langle v, x^{\prime} - x \rangle \leq \frac{1}{2 \rho (1- \gamma)^2}\Vert x^{\prime}-x\Vert^2 + d_{S}(x') - d_{S}(x).$$
\end{enumerate}
\end{proposition}
Another prominent class of sets is that of subsmooth sets, which encompasses the concepts of convex and uniformly prox-regular sets (see \cite{MR2115366} and \cite[Chapter~8]{MR4659163} for a survey).
\begin{definition}
Let $S$ be a closed subset of $\mathcal{H}$. We say that $S$ is \textit{subsmooth} at $x_0 \in S$, if for every $\varepsilon>0$ there exists $\delta>0$ such that
\begin{equation}\label{subsmooth}
\left\langle\xi_2-\xi_1, x_2-x_1\right\rangle \geq-\varepsilon\left\|x_2-x_1\right\|,   
\end{equation}
whenever $x_1, x_2 \in \mathbb{B}\left[x_0, \delta\right] \cap S$ and $\xi_i \in N\left(S ; x_i\right) \cap \mathbb{B}$ for $i \in\{1,2\}$. The set $S$ is said \emph{subsmooth} if it is subsmooth at each point of $S$. We further say that $S$ is \textit{uniformly subsmooth}, if for every $\varepsilon>0$ there exists $\delta>0$, such that \eqref{subsmooth} holds for all $x_1, x_2 \in S$ satisfying $\left\|x_1-x_2\right\| \leq \delta$ and all $\xi_i \in N\left(S ; x_i\right) \cap \mathbb{B}$ for $i \in\{1,2\}$. Let $(S(t))_{t\in I}$ be a family of closed sets of $\H$ indexed by a nonempty set $I$. The family is called \textit{equi-uniformly subsmooth}, if for all $\varepsilon>0$, there exists $\delta>0$ such that for all $t\in I$, the inequality \eqref{subsmooth} holds for all $x_1,x_2\in S(t)$ satisfying $\|x_1-x_2\|\leq \delta$ and all $\xi_i\in N(S(t);x_i)\cap \mathbb{B}$ with $i\in\{1,2\}$. 
\end{definition}
Given an interval $\mathcal{I}$, a set-valued map $F\colon \mathcal{I}\tto \H$ is called measurable if, for every open set $U$ of $\H$, the inverse image $F^{-1}(U) = \{t\in\mathcal{I}:F(t)\cap U\neq\emptyset\}$ is a Lebesgue measurable set. Whenever $\H$ is separable and $F$ takes nonempty and closed values, this definition is equivalent to the $\mathcal{L}\otimes \mathcal{B}(\H)$-measurability of the graph  $\operatorname{gph} F := \{(t,x)\in \mathcal{I}\times \H: x\in F(t)\}$  (see, e.g., \cite[Theorem 6.2.20]{MR2527754}). \\
\noindent A set-valued map $F\colon \H\tto \H$ is called upper semicontinuous from $\H$ into $\H_w$ if,  for every weakly closed set $C\subset \H$, the inverse image $F^{-1}(C)$ is a closed set of $\H$. If $F$ is upper semicontinuous, it is well-known (see \cite[Proposition 6.1.15 (c)]{MR2527754}) that the map $x\mapsto \sigma_{F(x)}(\xi)$ is upper semicontinuous for all $\xi\in \H$. Moreover, when $F$ takes convex and weakly compact values, these two properties are equivalent (see \cite[Proposition 6.1.17]{MR2527754}).\\
\noindent The \textit{projection} onto $S \subset \H$ at $x \in \H$ is the (possibly empty) set defined as
$$\operatorname{Proj}_{S}(x) := \{ z \in S: d_{S}(x) = \Vert x-z\Vert \}.$$
Whenever the projection set is a singleton, we denote it simply as $\operatorname{proj}_S(x)$. In most applications, the projection is difficult to calculate numerically, and one must resort to approximations of this object. The first of them, defined for $\varepsilon > 0$ and studied in \cite{MR2731287}, is the so-called set of $\varepsilon$-\emph{approximate projections}, given by
$$\proj_{S}^{\varepsilon}(x) := \{z \in S: \Vert x-z\Vert^{2} < d_{S}^2(x) + \varepsilon\},$$
which is always nonempty and open. It is clear that an element of the above set can be obtained through an appropriate optimization numerical method. However, the above definition requires that every approximate projection lies entirely within the set $S$. Hence, only optimization algorithms that ensure this condition can be used to obtain an approximate projection. \\
\noindent It is known that for any $x \in \H$ where $\operatorname{Proj}_{S}(x) \neq \emptyset$, the following formula holds:
\begin{equation}\label{formula-1}
x - z \in d_S(x) \partial_P d_S(z) \quad \textrm{ for all } z \in \operatorname{Proj}_S(x).
\end{equation}
The next result, proved in \cite[Lemma~1]{MR2731287}, provides an approximate version of the above formula without the nonemptiness assumption on the projection.
\begin{lemma}\label{aprox-proj-I}
Let $S \subset \H$ be a nonempty and closed set, $x \in \H$ and $\varepsilon > 0$. For each $z \in \proj_{S}^{\varepsilon}(x)$ there is $v \in \proj_{S}^{\varepsilon}(x)$ such that $\Vert z-v \Vert < 2 \sqrt{\varepsilon}$ and
$$x-z \in (4 \sqrt{\varepsilon} + d_S(x)) \partial_{P} d_S(v) + 3 \sqrt{\varepsilon} \B.$$
\end{lemma}
\noindent Now, we define the main object of this paper. 
\begin{definition}
Given a set $S\subset \H$, $x\in \H$, and $\varepsilon, \eta>0$, the set of $\varepsilon-\eta$ \textit{approximate projections} is defined as
$$\proj_S^{\varepsilon, \eta}(x):= \{  z \in S_{\eta}: \Vert x-z\Vert ^{2} < d_S^2(x) + \varepsilon \},$$
where $S_{\eta} \subset \H$ is any closed set such that $S \subset S_{\eta} \subset S + \eta \B$.
\end{definition}
\begin{lemma}  Let $\eta>0$ and $S_{\eta} \subset \H$ be any closed set such that $S \subset S_{\eta} \subset S + \eta \B$. Then, 
$$d_{S_{\eta}}(x) \leq d_{S}(x) \leq d_{S_{\eta}}(x) + \eta \quad \textrm{ for all } x\in \H.
$$
\end{lemma}
\begin{proof} Fix $x\in \H$ and let $\eta > 0$.  Since $S \subset S_{\eta} \subset S + \eta \B$, we obtain that 
$$d_{S + \eta \B}(x) \leq d_{S_{\eta}}(x) \leq d_{S}(x),$$
which proves the first inequality. To prove the second inequality, we observe that any $x \in S + \eta \B$ can be written as $x=s+\eta b$ for some $s\in S$ and $b\in \mathbb{B}$. Hence, $d_S(x)\leq \Vert x-s\Vert \leq \eta$. Moreover, if $x \notin S + \eta \B$, then, according to \cite[Lemma~2.1]{MR3025303},  $d_{S + \eta \B}(x) = d_S(x) - \eta$, which implies the result.  
\end{proof}
\begin{proposition}\label{approx-eta}
Let $\varepsilon, \eta>0$, and assume that $S\subset S_{\eta}\subset S+\eta \mathbb{B}$. Then,  
\begin{equation*}
\operatorname{proj}_S^{\varepsilon}(x)\subset \operatorname{proj}_S^{\varepsilon,\eta}(x)\subset \operatorname{proj}_S^{\varepsilon+2\eta (d_S(x)+\sqrt{\varepsilon})+\eta^2}(x)+\eta \mathbb{B} \quad \textrm{ for all } x\in \H. 
\end{equation*}
\end{proposition}
\begin{proof}
Fix $x\in \H$. The first assertion follows directly from the definition of the $\varepsilon-\eta$ approximate projection. To prove the second inclusion, let $z \in \proj_{S}^{\varepsilon, \eta}(x)$. Then $$\Vert x - z \Vert^{2} < d_{S}^{2}(x) + \varepsilon.$$
Since $z \in S_{\eta} \subset S + \eta \B$, there exists $s \in S, b \in \B$ such that $z = s + \eta b$. We observe that
$$\Vert x - s \Vert^{2} < d_{S}^{2}(x) + \varepsilon+2\eta (d_S(x)+\sqrt{\varepsilon})+\eta^2.$$
Hence, $s \in \proj_{S}^{\varepsilon+2\eta (d_S(x)+\sqrt{\varepsilon})+\eta^2}(x)$, which implies that $z \in \proj_S^{\varepsilon+2\eta (d_S(x)+\sqrt{\varepsilon})+\eta^2}(x)+\eta \mathbb{B}$.
\end{proof}
\noindent The next result provides a generalized version of Lemma \ref{aprox-proj-I} for $\varepsilon-\eta$ approximate projections.
\begin{lemma}\label{aprox-proj-II}
Let $S \subset \H$ be a nonempty, closed set, $x \in \H$, and $\varepsilon, \eta > 0$. Then, for each $z \in \proj_S^{\varepsilon, \eta}(x)$, there exists $v \in \proj_{S}^{\varepsilon'}(x)$ such that $\Vert z - v \Vert < 2 \sqrt{\varepsilon'} + \eta$ and
$$x - z \in (4 \sqrt{\varepsilon'} + d_{S}(x)) \partial_{P}d_{S}(v) + (3 \sqrt{\varepsilon'} + \eta) \B,$$
where $\varepsilon':= \varepsilon + 2 \eta(d_S(x) + \sqrt{\varepsilon}) + \eta^{2}$.
\end{lemma}
\begin{proof}
Fix $x \in \H$ and $\varepsilon, \eta > 0$, let $z \in \proj_{S}^{\varepsilon, \eta}(x)$.  Then,  by Proposition \ref{approx-eta}, there exist $s \in \proj_{S}^{\varepsilon'}(x)$ and $b \in \B$ such that $z = s + \eta b$ with $\varepsilon':= \varepsilon + 2 \eta(d_S(x) + \sqrt{\varepsilon}) + \eta^{2}$. Then, according to Lemma \ref{aprox-proj-I}, there is $v \in \proj_{S}^{\varepsilon'}(x)$ such that $\Vert s - v \Vert < 2 \sqrt{\varepsilon'}$ and 
$$ x- s \in (4 \sqrt{\varepsilon'} + d_{S}(x)) \partial_{P} d_{S}(v) + 3 \sqrt{\varepsilon'} \B.$$
Since $z - s \in \eta \B$, we observe that 
$$\Vert z-v \Vert \leq \Vert z-s \Vert + \Vert s-v \Vert < 2 \sqrt{\varepsilon'} + \eta.$$
Then, because $x - z = (x - s) + (s - z)$, we get that 
$$x-z \in (4 \sqrt{\varepsilon'} + d_{S}(x)) \partial_{P} d_{S}(v) + 3 \sqrt{\varepsilon'} \B + \eta \B,$$
which proves the desired result.
\end{proof}
The following result provides two important properties of $\varepsilon-\eta$ projections. Note that the second statement corresponds to a generalization of property (b) in Proposition \ref{prop-prox-regular}.
\begin{proposition}
    Let $S \subset \H$ be a $\rho$-uniformly prox-regular set. Then, one has:
    \begin{enumerate}[(a)]
    \item Let $x_n\to x \in U_{\rho}(S)$. Then for any $(z_n)$ and any pair of sequences of positive numbers $(\varepsilon_n)$ and $(\eta_n)$ converging to 0 with $z_n \in \proj_{S}^{\varepsilon_n, \eta_n}(x_n)$ for all $n \in \N$, we have that $z_n \to \proj_{S}(x)$.
    \item Let $\gamma \in ]0,1[$ and  $\eta \in ]0,\rho[$. Assume that $\eta\in ]0,\eta_0[$ and $\varepsilon \in ]0,\varepsilon_0]$, where $\eta_0$ and $\varepsilon_0$ are such that
    \begin{equation*}
    \gamma + 4\beta_0 \left(1 + \frac{1}{\rho} \right) + \frac{3 \eta_0}{2}
     + \left( \frac{4\beta_0}{\rho} + \gamma \right) \left(4\beta_0 + 2 \eta_0\right) = 1 , 
    \end{equation*}
    where $\beta_0 := \sqrt{\varepsilon_0}+\eta_0+\sqrt{2\eta_0\gamma \rho}$.
    Then, for all $z_i \in \proj_{S}^{\varepsilon, \eta}(x_i)$ and $x_i \in U_{\rho}^{\gamma}(S)$ for $i \in \{1,2\}$, we have
    $$(1-\digamma) \Vert z_1 -z_2 \Vert^{2} \leq \left( \sqrt{\varepsilon^{\prime}} + \frac{\eta}{2} \right) \Vert x_1 - x_2 \Vert^{2} + M \sqrt{\varepsilon^{\prime}} + N \eta + \langle x_1- x_2, z_1-z_2 \rangle,$$
    where $\digamma:= \frac{\alpha}{\rho}+4\sqrt{\varepsilon'}(1 + \frac{1}{\rho}) + \frac{3 \eta}{2} + \left(\frac{4\sqrt{\varepsilon'} + \alpha}{\rho}\right) \left(4 \sqrt{\varepsilon'} + 2 \eta\right)$, $\alpha:= \max \{d_{S} (x_1), d_{S}(x_2)\}$, $\varepsilon^{\prime}:=\varepsilon+2\eta(\alpha + \sqrt{\varepsilon}))+\eta^2$, $ M:= \left(\frac{4\sqrt{\varepsilon'} + \alpha}{\rho}\right)(16 \sqrt{\varepsilon'} + 16 \eta + 4)+24 \sqrt{\varepsilon^{\prime}} + 20 \eta + 11$ and $N := \left(\frac{4\sqrt{\varepsilon'} + \alpha}{\rho}\right)(4 \eta + 2) + 4\eta + 5$.
    \end{enumerate}
\end{proposition}
\begin{proof}
\textit{(a)}: We observe that for all $n \in \N$
$$\Vert z_n \Vert \leq \Vert z_n - x_n \Vert + \Vert x_n \Vert \leq d_{S}(x_n) + \sqrt{\varepsilon_n} + \Vert x_n \Vert.$$
Hence, since $\varepsilon_n \to 0$ and $x_n \to x$, we obtain $(z_n)$ is bounded. On the other hand, since $x \in U_{\rho}(S)$, the projection $\proj_{S}(x)$ is well-defined and
\begin{equation*}
\begin{aligned}
\Vert z_n - \proj_{S}(x) \Vert^{2} &= \Vert z_n - x_n \Vert^{2} - \Vert x_n - \proj_{S}(x) \Vert^{2}\\
&+ 2 \langle x - \proj_S(x), z_n - \proj_{S}(x) \rangle + 2 \langle z_n - \proj_S(x), x_n - x \rangle\\
&\leq d_S^{2}(x_n) + \varepsilon_n - \Vert x_n - \proj_{S}(x) \Vert^{2}\\
& + 2 \langle x - \proj_S(x), z_n - \proj_{S}(x) \rangle + 2 \langle z_n - \proj_S(x), x_n - x \rangle\\
&\leq \varepsilon_n + 2 \langle x - \proj_S(x), z_n - \proj_{S}(x) \rangle + 2 \langle z_n - \proj_S(x), x_n - x \rangle\\
&\leq \varepsilon_n + 2 \langle x - \proj_S(x), z_n - \proj_{S}(x) \rangle + 2 \Vert z_n - \proj_S(x)\Vert \cdot \Vert x_n - x \Vert,
\end{aligned}
\end{equation*}
where we have used that $z_n \in \proj_{S}^{\varepsilon_n, \eta_n}(x_n)$ and the fact that $d_S^2(x_n) \leq \Vert x_n - \proj_S(x) \Vert^{2}$. On the other hand, since $z_n \in S_{\eta_n} \subset S + \eta_n \B$, we observe that there exists $s_n \in S$ and $b_n \in \B$ such that $z_n = s_n + \eta_n b_n$. Hence, 
\begin{equation*}
\begin{aligned}
2\langle x - \proj_S(x), z_n - \proj_{S}(x) \rangle &= 2\langle x - \proj_S(x), s_n - \proj_{S}(x) \rangle+ 2\langle x - \proj_S(x), \eta_n b_n \rangle.
\end{aligned}
\end{equation*}
Moreover, according to inclusion \eqref{formula-1} and the $\rho$-uniform prox-regularity of $S$, we obtain that
\begin{equation*}
    \begin{aligned}
    2\langle x - \proj_S(x), s_n - \proj_{S}(x) \rangle &\leq \frac{d_S(x)}{\rho} \Vert s_n - \proj_S(x) \Vert^{2}\\
     &=\frac{d_S(x)}{\rho} \Vert z_n-\eta_n b_n - \proj_S(x) \Vert^{2}\\
        &\leq \frac{d_S(x)}{\rho} \left( \Vert z_n-\proj_S(x)\Vert^2+2\eta_n \Vert z_n-\proj_S(x)\Vert+  \eta_n^2\right), 
    \end{aligned}
\end{equation*}
where we have used that $z_n = s_n + \eta_n b_n$. Therefore, 
\begin{equation*}
    \begin{aligned}
    \Vert z_n - \proj_{S}(x) \Vert^{2} &\leq \varepsilon_n +\frac{d_S(x)}{\rho} \left( \Vert z_n-\proj_S(x)\Vert^2+2\eta_n \Vert z_n-\proj_S(x)\Vert+\eta_n^2\right)\\&+2\eta_n d_S(x)+2\Vert  z_n - \proj_S(x)\Vert \cdot \Vert x_n - x \Vert.
    \end{aligned}
\end{equation*}
Rearranging terms, we obtain that
\begin{equation*}
\begin{aligned}
\Vert z_n - \proj_{S}(x) \Vert^{2} &\leq  \frac{\rho \varepsilon_n}{\rho - d_{S}(x)} + \frac{d_S(x)}{\rho - d_{S}(x)} \left( 2\eta_n \Vert z_n-\operatorname{proj}_S(x)\Vert +\eta_n^2 \right)\\
&+\frac{\rho}{\rho-d_S(x)}\left(2\eta_n d_S(x)+2\Vert z_n-\operatorname{proj}_S(x)\Vert \cdot \Vert x_n-x\Vert \right).
\end{aligned}
\end{equation*}
Finally, since $x_n \to x$, $(z_n)$ is bounded, $\varepsilon_n \to 0$ and $\eta_n \to 0$, we concluded that $z_n \to \proj_{S}(x)$.\\
\noindent \textit{(b)}: For $i=1,2$, let  $z_i \in \proj_{S}^{\varepsilon, \eta}(x_i)$. By virtue of Lemma \ref{aprox-proj-II}, there exist $v_i, b_i \in \H$ for $i \in  \{1,2\}$ such that
$$b_i \in \B, v_i \in \proj_S^{\varepsilon_{i}'}(x_i), \Vert z_i - v_i \Vert \leq 2\sqrt{\varepsilon_{i}'} + \eta \text{ and } \frac{x_i - z_i -(3 \sqrt{\varepsilon_{i}'} + \eta)b_i}{4\sqrt{\varepsilon_{i}'} + d_{S}(x_i)} \in \partial_{P} d_{S}(v_i),$$
where $\varepsilon_{i}^{\prime}:= \varepsilon + 2\eta(d_{S}(x_i) + \sqrt{\varepsilon}) + \eta^{2}$. Hence, for $i \in \{1,2\}$, one has
$$x_i - z_i - (3\sqrt{\varepsilon_{i}'} + \eta)b_i \in N^{P}(S;v_i) \cap \tau \B \text{ for all } \tau \geq 4 \sqrt{\varepsilon_{i}'} + d_{S}(x_i).$$
For $i \in \{1,2\}$, let us consider
$$\zeta_i := \frac{x_i - z_i - (3\sqrt{\varepsilon_{i}'} + \eta)b_i}{4\sqrt{\varepsilon'} + \alpha} \in N^{P}(S;v_i) \cap \B ,$$
where $\alpha:= \max \{d_{S}(x_1), d_{S}(x_2)\}$ and $\varepsilon' = \max \{\varepsilon_{1}', \varepsilon_{2}'\} = \varepsilon + 2 \eta(\alpha + \sqrt{\varepsilon}) + \eta^{2}$. Since $S$ is $\rho-$uniformly prox-regular and $v_i \in S$, using the hypomonotonicity of the truncated proximal normal cone (see Proposition \ref{prop-prox-regular} (c)), we obtain that
\begin{equation}\label{ineqb-1}
\langle \zeta_1 - \zeta_2, v_1 -v_2  \rangle \geq -\frac{1}{\rho}\Vert v_1 - v_2 \Vert^{2} .
\end{equation}
On the one hand, since $\Vert z_i - v_i \Vert \leq 2 \sqrt{\varepsilon_{i}'} + \eta$ for $i \in \{1,2\}$, we get that
\begin{equation}\label{ineqb-2}
\Vert v_1 - v_2 \Vert \leq \Vert v_1 - z_1 \Vert + \Vert z_1 - z_2 \Vert + \Vert z_2 - v_2 \Vert \leq 4 \sqrt{\varepsilon'} + 2\eta + \Vert z_1 - z_2 \Vert ,
\end{equation}
and for all $z \in \H$ and $i \in \{1,2\}$, one has
\begin{equation}\label{ineqb-3}
\vert \langle z, v_i - z_i \rangle \vert \leq (\sqrt{\varepsilon'} + \frac{\eta}{2})\frac{\Vert z \Vert^{2}}{2} + \frac{\Vert z_i - v_i \Vert^{2}}{2(\sqrt{\varepsilon'} + \frac{\eta}{2})} \leq (\sqrt{\varepsilon'} + \frac{\eta}{2}) \frac{\Vert z \Vert^{2}}{2} + 2 \sqrt{\varepsilon'} + \eta .
\end{equation}
Due to \eqref{ineqb-2} we get that
\begin{equation*}
    \begin{aligned}
    &\langle (x_1 - z_1 - (3 \sqrt{\varepsilon_{1}'} + \eta)b_1) -(x_2-z_2 - (3 \sqrt{\varepsilon_{2}'} + \eta)b_2), v_1 - v_2 \rangle\\
    =& \ \langle (3\sqrt{\varepsilon_{2}'} + \eta)b_2 - (3\sqrt{\varepsilon_{1}'} + \eta)b_1, v_1 - v_2 \rangle + \langle x_1 - x_2, v_1 - v_2 \rangle - \langle z_1 - z_2, v_1 - v_2 \rangle\\
    \leq & \ 24\varepsilon' + 20\sqrt{\varepsilon'}\eta + 4\eta^{2} + (6\sqrt{\varepsilon'} + 2\eta)\Vert z_1 - z_2 \Vert + \langle x_1 - x_2, v_1 - v_2 \rangle - \langle z_1 - z_2, v_1 - v_2 \rangle\\
    = & \ 24\varepsilon' + 20\sqrt{\varepsilon'}\eta + 4\eta^{2} + (6\sqrt{\varepsilon'} + 2\eta)\Vert z_1 - z_2 \Vert +\langle x_1 - x_2, v_1 - z_1 \rangle + \langle x_1 - x_2, z_1 - z_2 \rangle \\
    &+ \langle x_1 - x_2, z_2 - v_2 \rangle- \langle z_1 - z_2, v_1 - z_1 \rangle - \Vert z_1 - z_2 \Vert^{2} - \langle z_1 - z_2, z_2 - v_2 \rangle\\
    \leq & \ 24\varepsilon' + 20\sqrt{\varepsilon'}\eta + 4\eta^{2} + (6\sqrt{\varepsilon'} + 2\eta)\Vert z_1 - z_2 \Vert + \langle x_1 - x_2, z_1 - z_2 \rangle\\
    &+ (\sqrt{\varepsilon'} + \frac{\eta}{2}) \Vert x_1 - x_2 \Vert^{2} + (\sqrt{\varepsilon'} + \frac{\eta}{2}) \Vert z_1 - z_2 \Vert^{2} + 8 \sqrt{\varepsilon'} + 4 \eta - \Vert z_1 - z_2 \Vert^{2} ,
    \end{aligned}
\end{equation*}
where we have used \eqref{ineqb-3} in the last inequality with $z = x_1 - x_2$ and $z = z_1 - z_2$. Next, by noting that $(6\sqrt{\varepsilon'} + 2 \eta) \Vert z_1 - z_2 \Vert \leq 3 \sqrt{\varepsilon'} + \eta + ( 3 \sqrt{\varepsilon'} + \eta)\Vert z_1 - z_2 \Vert^{2}$, we obtain that
\begin{equation*}
\begin{aligned}
&\langle (x_1 - z_1 - (3 \sqrt{\varepsilon_{1}'} + \eta)b_1) -(x_2-z_2 - (3 \sqrt{\varepsilon_{2}'} + \eta)b_2), v_1 - v_2 \rangle\\
\leq & \ 24\varepsilon' + 20\sqrt{\varepsilon'}\eta + 4\eta^{2} + 11 \sqrt{\varepsilon'} + 5 \eta + (\sqrt{\varepsilon'} + \frac{\eta}{2}) \Vert x_1 - x_2 \Vert^{2}\\
& + \langle x_1 - x_2, z_1 - z_2 \rangle - (1 - 4 \sqrt{\varepsilon'} - \frac{3 \eta}{2})\Vert z_1 - z_2 \Vert^{2} .
\end{aligned}
\end{equation*}
Therefore, from inequality \eqref{ineqb-1} and the above calculations, it follows that
\begin{equation*}
    \begin{aligned}
    \hspace{-3mm} -\frac{4\sqrt{\varepsilon^{\prime}}+\alpha}{\rho}\Vert v_1-v_2\Vert^2&\leq  (4\sqrt{\varepsilon^{\prime}}+\alpha)\langle \zeta_1-\zeta_2,v_1-v_2\rangle\\
    &= \langle (x_1 - z_1 - (3 \sqrt{\varepsilon_{1}'} + \eta)b_1) -(x_2-z_2 - (3 \sqrt{\varepsilon_{2}'} + \eta)b_2), v_1 - v_2 \rangle\\
    \leq & \ 24\varepsilon' + 20\sqrt{\varepsilon'}\eta + 4\eta^{2} + 11 \sqrt{\varepsilon'} + 5 \eta + (\sqrt{\varepsilon'} + \frac{\eta}{2}) \Vert x_1 - x_2 \Vert^{2}\\
& + \langle x_1 - x_2, z_1 - z_2 \rangle - (1 - 4 \sqrt{\varepsilon'} - \frac{3 \eta}{2})\Vert z_1 - z_2 \Vert^{2} .
    \end{aligned}
\end{equation*}
Finally, by using that $\Vert v_1 - v_2 \Vert^{2} \leq 16 \varepsilon' + 16 \sqrt{\varepsilon'}\eta +4 \eta^{2} + 4\sqrt{\varepsilon'} + 2 \eta + (1 +4\sqrt{\varepsilon'} +2\eta) \Vert z_1 - z_2 \Vert^{2}$, we get that
\begin{equation*}
\begin{aligned}
    &\left[ 1-\frac{\alpha}{\rho}-4\sqrt{\varepsilon'} \left(1 + \frac{1}{\rho}\right) - \frac{3 \eta}{2} - \left(\frac{4\sqrt{\varepsilon'} + \alpha}{\rho}\right) \left(4 \sqrt{\varepsilon'} + 2 \eta\right)\right] \Vert z_1-z_2\Vert^2\\
    &\leq  \left( \sqrt{\varepsilon^{\prime}} + \frac{\eta}{2} \right) \Vert x_1 - x_2 \Vert^{2} +  \langle x_1 - x_2, z_1 - z_2 \rangle + \left( \frac{4 \sqrt{\varepsilon'} + \alpha}{\rho}\right) \left( 16 \varepsilon' + 16 \sqrt{\varepsilon'}\eta + 4 \eta^{2} + 4 \sqrt{\varepsilon'} + 2\eta\right)\\
    & \hspace{25mm} + 20\sqrt{\varepsilon'}\eta   + 24\varepsilon' +  11 \sqrt{\varepsilon'} + 4\eta^{2} + 5 \eta ,
\end{aligned}
\end{equation*}
which proves the desired inequality.
\end{proof}

\section{Inexact Catching-Up Algorithm for Sweeping Processes}\label{Sec-3}

In this section, based on the concept of $\varepsilon-\eta$ projection, we propose an inexact  
 catching-up algorithm for the existence of solutions to the sweeping process:
\begin{equation}\label{SP}
\left\{
\begin{aligned}
\dot{x}(t) &\in -N(C(t) ; x(t))+F(t, x(t)) \quad \textrm{ a.e. } t \in[0, T], \\
x(0) & =  x_0 \in C(0),
\end{aligned}\right.
\end{equation}
where $C\colon [0, T] \rightrightarrows \mathcal{H}$ is a set-valued map with closed values in a Hilbert space $\mathcal{H}$, $N(C(t) ; x)$ stands for the Clarke normal cone to $C(t)$ at $x$, and $F\colon [0, T] \times \mathcal{H} \rightrightarrows \mathcal{H}$ is a given set-valued map with nonempty closed and convex values. \\
\noindent The proposed algorithm is given as follows. For $n \in \mathbb{N}^*$, let $\left(t_k^n: k=0,1, \ldots, n\right)$ be a uniform partition of $[0, T]$ with uniform time step $\mu_n:=T / n$. Let $\left(\varepsilon_n, \eta_n\right)$ be a sequence of positive numbers such that $\varepsilon_n / \mu_n^2 \rightarrow 0$ and $\eta_n/\mu_n \to 0$. We consider a sequence of piecewise continuous linear approximations $\left(x_n\right)$ defined as $x_n(0)=x_0$ and for any $k \in\{0, \ldots, n-1\}$ and $t \in ] t_k^n, t_{k+1}^n ]$
\begin{equation}\label{piecewise_construction}
x_n(t)=x_k^n+\frac{t-t_k^n}{\mu_n}\left(x_{k+1}^n-x_k^n-\int_{t_k^n}^{t_{k+1}^n} f\left(s, x_k^n\right)\mathrm{d} s\right)+\int_{t_k^n}^t f\left(s, x_k^n\right) \mathrm{d} s,
\end{equation}
where $x_0^n=x_0$ and
\begin{equation}\label{approx_proj_step}
x_{k+1}^n \in \proj_{C\left(t_{k+1}^n\right)}^{\varepsilon_n, \eta_n}\left(x_k^n+\int_{t_k^n}^{t_{k+1}^n} f\left(s, x_k^n\right) \mathrm{d} s\right) \text { for } k \in\{0,1, \ldots, n-1\} .
\end{equation}
Here $f(t, x)$ denotes any selection of $F(t, x)$ such that $f(\cdot, x)$ is measurable for all $x \in \mathcal{H}$. For simplicity, we consider $f(t, x) \in \operatorname{proj}_{F(t, x)}^\gamma(0)$ for some $\gamma>0$.

The above algorithm will be called \emph{inexact catching-up algorithm} because the projection is not necessarily exactly calculated. We will prove that the above algorithm converges for several families of moving sets as long as inclusion \eqref{approx_proj_step} is verified.
Let us consider functions $\delta_n(\cdot)$ and $\theta_n(\cdot)$ defined as
$$
\delta_n(t):= \begin{cases}
t_k^n & \text { if } t \in[t_k^n, t_{k+1}^n[\\
t_{n-1}^n & \text { if } t=T,
\end{cases} \text { and } \theta_n(t):= \begin{cases}t_{k+1}^n & \text { if } t \in [t_k^n, t_{k+1}^n[ \\
T & \text { if } t=T .\end{cases}.
$$
In what follows, we show useful properties of the above algorithm, which will help in proving the existence of solutions for the sweeping process \eqref{SP} in three cases:
\begin{enumerate}[(i)]
    \item The map $t \rightrightarrows C(t)$ takes uniformly prox-regular values.
    \item The map $t \rightrightarrows C(t)$ takes subsmooth and ball-compact values.
    \item $C(t) \equiv C$ in $[0, T]$ and $C$ is ball-compact.
\end{enumerate}
In what follows, $F\colon [0, T] \times \mathcal{H} \rightrightarrows \mathcal{H}$ will be a set-valued map with nonempty, closed, and convex values. Moreover, we will consider the following conditions:
\begin{itemize}
\item[$\left(\H_1^F\right)$] For all $t \in[0, T], F(t, \cdot)$ is upper semicontinuous from $\H$ into $\H_w$.
\item[$\left(\H_2^F\right)$] There exists $h\colon \mathcal{H} \rightarrow \mathbb{R}^{+}$ Lipschitz continuous (with constant $L_h>0$) such that
$$
d(0, F(t, x)):=\inf \{\|w\|: w \in F(t, x)\} \leq h(x) \quad \textrm{ for all }x\in \mathcal{H} \textrm{ and a.e. } t\in [0,T]. 
$$
\item[$\left(\H_3^F\right)$] There is $\gamma>0$ such that the set-valued map $(t, x) \rightrightarrows \operatorname{proj}_{F(t, x)}^\gamma(0)$ has a selection $f\colon [0, T] \times \mathcal{H} \rightarrow \mathcal{H}$ with $f(\cdot, x)$ is measurable for all $x \in \mathcal{H}$.
\end{itemize}
\noindent The following proposition, proved in \cite{MR2731287}, provides a condition for the feasibility of hypothesis $\left(\H_3^F\right)$.
\begin{proposition}
    Let us assume that $\H$ is a separable Hilbert space. Moreover we suppose $F(\cdot,x)$ is measurable for all $x\in\H$, then $\left(\H_3^F\right)$ holds for all $\gamma>0$.
\end{proposition}
\noindent Now, we establish the main properties of the inexact catching-up algorithm. 
\begin{theorem}\label{scheme}
Assume, in addition to $\left(\mathcal{H}_1^F\right),\left(\mathcal{H}_2^F\right)$ and $\left(\mathcal{H}_3^F\right)$, that $C\colon [0, T] \rightrightarrows \mathcal{H}$ is a set-valued map with nonempty and closed values such that
\begin{equation}\label{lip_movingset}
\operatorname{Haus}(C(t), C(s)) \leq L_C|t-s| \text { for all } t, s \in[0, T] .
\end{equation}

Then, the sequence of functions $\left(x_n\colon [0, T] \rightarrow \mathcal{H}\right)$ generated by numerical scheme \eqref{piecewise_construction} and \eqref{approx_proj_step} satisfies the following properties:
\begin{enumerate}[(a)]
\item There are non-negative constants $K_1, K_2, K_3, K_4, K_5$ such that for all $n \in \N$ and $t \in [0, T]$:
\begin{enumerate}[(i)]
    \item $d_{C(\theta_n(t))}(x_n(\delta_n(t)) + \int_{\delta_n(t)}^{\theta_n(t)} f(s, x_n(\delta_n(t))) \mathrm{d}s) \leq (L_C + h(x(\delta_n(t))) + \sqrt{\gamma}) \mu_n + \eta_n$.
    \item $\Vert x_n(\theta_n(t)) - x_0 \Vert \leq K_1$.
    \item $\Vert x_n(t) \Vert \leq K_2$.
    \item $d_{C(\theta_n(t))}(x_n(\delta_n(t)) + \int_{\delta_n(t)}^{\theta_n(t)} f(s, x_n(\delta_n(t))) \mathrm{d}s) \leq K_3 \mu_n + \eta_n$.
    \item $\Vert x_n(\theta_n(t)) - x_n(\delta_n(t)) \Vert \leq K_4 \mu_n + \sqrt{\varepsilon_n} + \eta_n$.
    \item $\Vert x_n(t) - x_n(\theta_n(t)) \Vert \leq K_5 \mu_n + 2 \sqrt{\varepsilon_n} +2 \eta_n$.
\end{enumerate}
\item There exists $K_6 > 0$ such that for all $t \in [0, T]$ and $m,n \in \N$ we have
$$d_{C(\theta_n(t))}(x_m(t)) \leq K_6 \mu_m + L_C \mu_n + 2 \sqrt{\varepsilon_m} + 3 \eta_m .$$
\item There exists $K_7 > 0$ such that for all $n \in \N$
$$\Vert \dot{x}_n(t) \Vert \leq K_7 \textrm{ a.e. } t \in [0,T] .$$
\item For all $n \in \N$ and $k \in \{0, 1, \ldots, n-1\}$, there is $v_{k+1}^{n} \in C(t_{k+1}^{n})$ such that for all $t \in ]t_k^{n}, t_{k+1}^{n}[$:
\begin{equation}\label{xn-dot-inclusion}
\dot{x}_n(t) \in - \frac{\lambda_n(t)}{\mu_n} \partial_{P}d_{C(\theta_n(t))}(v_{k+1}^{n}) + f(t, x_n(\delta_n(t))) + \frac{(3\sqrt{\sigma_n} + \eta_n)}{\mu_n} \B .
\end{equation}
where $\lambda_n(t) = 4 \sqrt{\sigma_n} + (L_C + h(x_n(\delta_n(t))) + \sqrt{\gamma})\mu_n + \eta_n$ and $\sigma_n = 2 \varepsilon_n + 2K_3 \eta_n \mu_n + 4\eta_n^{2}$. Moreover, $\Vert v_{k+1}^{n} - x_n(\theta_n(t)) \Vert < 2 \sqrt{\sigma_n} + \eta_n.$
\end{enumerate}
\end{theorem}
\begin{proof}
\textit{(a)}: Set $\mu_n := T/n$ and let $(\varepsilon_n)$ and $(\eta_n)$ be sequences of non-negative numbers such that $\varepsilon_n/\mu_n^{2} \rightarrow 0$ and $\eta_n/ \mu_n \to 0$. We define $\mathfrak{c} := \sup_{n \in \N} \frac{\sqrt{\varepsilon_n} + \eta_n}{\mu_n}$. We denote by $L_h$ the Lipschitz constant of $h$. For all $t \in [0, T]$ and $n \in \N$, we define $\tau_n(t) := x_n(\delta_n(t)) + \int_{\delta_n(t)}^{\theta_n(t)} f(s, x_n(\delta_n(t)) \mathrm{d}s$ since distances functions are 1-Lipschitz
\begin{equation}\label{bound1th1}
d_{C(\theta_n(t))}(\tau_n(t)) \leq d_{C(\theta_n(t))}(x_n(\delta_n(t))) + \Vert \int_{\delta_n(t)}^{\theta_n(t)} f(s, x_n(\delta_n(t))) \mathrm{d}s \Vert .
\end{equation}
On the one hand, by virtue of  \eqref{approx_proj_step}, we have that $x_k^{n} \in C(t_k^{n}) + \eta_n \B$, which implies that  for some $b_n(t) \in \B$, $x_n(\delta_n(t)) - \eta_n b_n(t) \in C(\delta_n(t))$. Then, 
\begin{equation*}
\begin{aligned}
d_{C(\theta_n(t))}(x_n(\delta_n(t))) &= d_{C(\theta_n(t))}(x_n(\delta_n(t))) - d_{C(\delta_n(t)}(x_n(\delta_n(t)) - \eta_n b_n(t))\\
&\leq d_{C(\theta_n(t))}(x_n(\delta_n(t))) - d_{C(\delta_n(t))}(x_n(\delta_n(t))) + \Vert \eta_n b_n(t) \Vert .
\end{aligned}
\end{equation*}
Then, by using \eqref{lip_movingset} and the fact that $\operatorname{Haus}(A,B) = \sup_{w \in \mathcal{H}} \vert d_A(w) - d_B(w) \vert$, we obtain that
\begin{equation*}
d_{C(\theta_n(t))}(x_n(\delta_n(t)))\leq L_C \mu_n + \eta_n.
\end{equation*}
On the other hand, since $f(t, x_n(\delta_n(t))) \in \proj_{F(t, x_n(\delta_n(t)))}^{\gamma}(0)$  and $\left(\H_2^F\right)$ holds, we get that 
$$\Vert f(s,x_n(\delta_n(t))) \Vert \leq h(x_n(\delta_n(t))) + \sqrt{\gamma}.$$ Then, it follows from \eqref{bound1th1} that
\begin{equation*}
\begin{aligned}
d_{C(\theta_n(t))}(\tau_n(t)) &\leq L_C \mu_n + \eta_n + \int_{\delta_n(t)}^{\theta_n(t)} (h(x_n(\delta_n(t))) + \sqrt{\gamma}) \mathrm{d}s\\
&\leq (L_C + h(x_n(\delta_n(t))) + \sqrt{\gamma})\mu_n + \eta_n ,
\end{aligned}
\end{equation*}
which proves $(i)$. Moreover, since $x_n(\theta_n(t)) \in \proj_{C(\theta_n(t))}^{\varepsilon_n, \eta_n}(\tau_n(t))$, we get that
\begin{equation}\label{bound2th1}
\begin{aligned}
\Vert x_n(\theta_n(t)) - \tau_n(t) \Vert & \leq d_{C(\theta_n(t))}(\tau_n(t))) + \sqrt{\varepsilon_n}\\
& \leq (L_C + h(x_n(\delta_n(t))) + \sqrt{\gamma})\mu_n + \eta_n + \sqrt{\varepsilon_n} ,
\end{aligned}
\end{equation}
which yields
\begin{equation}\label{bound3th1}
\begin{aligned}
\Vert x_n(\theta_n(t)) - x_n(\delta_n(t)) \Vert  & \leq (L_C + 2h(x_n(\delta_n(t))) + 2 \sqrt{\gamma}) \mu_n + \eta_n + \sqrt{\varepsilon_n}\\
& \leq (L_C + 2h(x_0) + 2L_h \Vert x_n(\delta_n(t)) - x_0 \Vert + 2\sqrt{\gamma})\mu_n + \eta_n + \sqrt{\varepsilon_n} ,
\end{aligned}
\end{equation}
where we have used that $h$ is Lipschitz continuous with constant $L_h > 0$ in the last inequality.
Hence for all $t \in [0, T]$
\begin{equation*}
\begin{aligned}
\Vert x_n(\theta_n(t)) - x_0 \Vert \leq & (1 + 2L_h \mu_n) \Vert x_n(\theta_n(t)) - x_0 \Vert\\
& + (L_C + 2h(x_0) + 2\sqrt{\gamma}) \mu_n + \eta_n + \sqrt{\varepsilon_n} .
\end{aligned}
\end{equation*}
The above inequality means that for all $k \in \{0, 1, \ldots, n-1\}:$
$$\Vert x_{k+1}^{n} - x_0 \Vert \leq (1 + 2L_h \mu_n) \Vert x_k^{n} - x_0 \Vert + (L_C + 2h(x_0) + 2\sqrt{\gamma}) \mu_n + \eta_n + \sqrt{\varepsilon_n} .$$
Then, by \cite[p.183]{MR1488695}, we obtain that for all $k \in \{0, \ldots, n-1\}$
\begin{equation}\label{bound4th1}
\begin{aligned}
\Vert x_{k+1}^{n} - x_0 \Vert \leq & \ (k+1)((L_C + 2h(x_0) + 2\sqrt{\gamma})\mu_n + \eta_n + \sqrt{\varepsilon_n}) \exp(2L_h(k+1)\mu_n)\\
\leq & \ T(L_C + 2h(x_0) + 2\sqrt{\gamma} + \mathfrak{c}) \exp(2L_h T) =: K_1 ,
\end{aligned}
\end{equation}
which proves $(ii)$.\\
\noindent $(iii)$: By definition of $x_n$, for $t \in ]t_k^n, t_{k+1}^n]$ and $k \in \{0, \ldots, n-1\}$, using \eqref{piecewise_construction}
\begin{equation*}
\begin{aligned}
\Vert x_n(t) \Vert \leq & \ \Vert x_k^n \Vert + \Vert x_{k+1}^n - \tau_n(t) \Vert + \int_{t_k^n}^n \Vert f(s, x_k^n) \Vert \mathrm{d}s\\
\leq & \ K_1 + \Vert x_0 \Vert + (L_C + \sqrt{\gamma} + h(x_k^n))\mu_n + \eta_n + \sqrt{\varepsilon_n} + (h(x_k^n) + \sqrt{\gamma})\mu_n ,
\end{aligned}
\end{equation*}
where we have used \eqref{bound2th1} and \eqref{bound4th1}. Moreover, it is clear that for $k \in \{0, \ldots, n\}$
$$h(x_k^n) \leq h(x_0) + L_h \Vert x_k^n -x_0 \Vert \leq h(x_0) + L_h K_1$$
Therefore, for all $t \in [0,T]$
\begin{equation*}
\begin{aligned}
\Vert x_n(t) \Vert \leq & \ K_1 + \Vert x_0 \Vert + (L_C + 2(h(x_0) + L_h K_1 + \sqrt{\gamma}))\mu_n + \eta_n + \sqrt{\varepsilon_n} \\
\leq & \ K_1 + \Vert x_0 \Vert + T(L_C + 2(h(x_0) +L_h K_1 + \sqrt{\gamma}) + \mathfrak{c}) := K_2 ,
\end{aligned}
\end{equation*}
which proves $(iii)$. \\
\noindent  $(iv)$ By using the Lipschitz continuity of $h$, we see that $h(x(\delta_n(t)) \leq  \ h(x_0) + L_h \Vert x_n(\delta_n) - x_0 \Vert$. Hence, by virtue of $(i)$ and $(iii)$, there exists $K_3 =  (L_C + h(x_0) + L_h(K_2 + \Vert x_0 \Vert) + \sqrt{\gamma}) > 0$ for which $(iv)$ holds for all $n \in \N$.\\
\noindent $(v)$: From \eqref{bound3th1} and \eqref{bound4th1} it is easy to see that there exists $K_4 > 0$ such that for all $n \in \N$ and $t \in [0,T]$: $\Vert x_n(\theta_n(t)) - x_n(\delta_n(t)) \Vert \leq K_4 \mu_n + \sqrt{\varepsilon_n} + \eta_n$.\\
\noindent $(vi)$: To conclude this part, we consider $t \in ]t_k^n, t_{k+1}^{n}]$ for some $k \in \{0, \ldots, n-1\}$.\\
Then $x_n(\theta_n(t)) = x_{k+1}^{n}$ and also
\begin{equation*}
\begin{aligned}
\Vert x_n(\theta_n(t)) - x_n(t) \Vert \leq & \ \Vert x_{k+1}^n - x_k^{n} \Vert + \Vert x_{k+1}^{n} - \tau_n(t) \Vert + \int_{t_{k}^{n}}^{t} \Vert f(s, x_k^{n}) \Vert \mathrm{d}s\\
\leq & \ K_4 \mu_n + \sqrt{\varepsilon_n} + \eta_n + (L_C + \sqrt{\gamma} + h(x_0) + L_h K_1)\mu_n + \sqrt{\varepsilon_n} + \eta_n\\
& \ + (h(x_k^{n}) + \sqrt{\gamma})\mu_n\\
\leq & \ \underbrace{(K_4 + L_C + 2(h(x_0) + L_h K_1) + 2\sqrt{\gamma}))}_{=: K_5} \mu_n + 2\sqrt{\varepsilon_n} + 2 \eta_n ,
\end{aligned}
\end{equation*}
and we conclude this first part.\\
\noindent 
\textit{(b)}: Let $m, n \in \N$ and $t \in [0,T]$, then
\begin{equation*}
\begin{aligned}
d_{C(\theta_n(t))}(x_m(t)) \leq & \ d_{C(\theta_n(t))}(x_m(\theta_m(t))) + \Vert x_m(\theta_m(t)) - x_m(t) \Vert\\
\leq & \ d_{C(\theta_n(t))}(x_m(\theta_m(t))) + K_5 \mu_m + 2\sqrt{\varepsilon_m} + 2\eta_m ,
\end{aligned}
\end{equation*}
where we have used $(v)$. Since $x_{k+1}^{m} \in C(t_{k+1}^{m}) + \eta_m \B$, we have that $x_m(\theta_m(t)) - \eta_m b_m(t) \in C(\theta_m(t))$ where $b_m(t) \in \B$, then we have 
\begin{equation*}
\begin{aligned}
d_{C(\theta_n(t))}(x_m(\theta_m(t))) = & \ d_{C(\theta_n(t))}(x_m(\theta_m(t))) - d_{C(\theta_m(t))}(x_m(\theta_m(t)) - \eta_m b_m(t)) \\
\leq & \ d_{C(\theta_n(t))}(x_m(\theta_m(t))) - d_{C(\theta_m(t))}(x_m(\theta_m(t))) + \eta_m \\
\leq & \ d_{H}(C(\theta_n(t)), C(\theta_m(t))) + \eta_m .
\end{aligned}
\end{equation*}
Therefore,
\begin{equation*}
\begin{aligned}
d_{C(\theta_n(t))}(x_m(t)) \leq & \ d_{H}(C(\theta_n(t)), C(\theta_m(t))) + \eta_m + K_5 \mu_m + 2\sqrt{\varepsilon_m} + 2\eta_m\\
\leq & \ L_C \vert \theta_n(t) - \theta_m(t) \vert + K_5 \mu_m + 2\sqrt{\varepsilon_m} + 3 \eta_m\\
\leq & \ L_C(\mu_n + \mu_m) + K_5 \mu_m + 2\sqrt{\varepsilon_m} + 3 \eta_m .
\end{aligned}
\end{equation*}
Hence, by setting $K_6 := K_5 + L_C$ we proved $(b)$. \\
\noindent \textit{(c)}: Let $n \in \N$, $k \in \{0, \ldots, n-1\}$ and $t \in ]t_k^n, t_{k+1}]$. Then,
\begin{equation*}
\begin{aligned}
\Vert \dot{x}_n(t) \Vert = & \ \Vert \mu_n^{-1} ( x_{k + 1} ^n - x_k^n - \int_{t_k^n}^{t_{k+1}^n} f(s, x_k^n) \mathrm{d}s) + f(t, x_k^n) \Vert \\
\leq & \ \frac{1}{\mu_n} \Vert x_n(\theta_n(t)) - \tau_n(t) \Vert + \Vert f(t,x_k^n) \Vert \\
\leq & \ \frac{1}{\mu_n}((L_C + h(x_k^n) + \sqrt{\gamma})\mu_n + \sqrt{\varepsilon_n} + \eta_n) + h(x_k^n) + \sqrt{\gamma}\\
\leq & \ \frac{\sqrt{\varepsilon_n} + \eta_n}{\mu_n} + L_C + 2(h(x_0) + L_h K_1 + \sqrt{\gamma})\\
\leq & \ \mathfrak{c} + L_C + 2(h(x_0) + L_h K_1 + \sqrt{\gamma}) =: K_7 ,
\end{aligned}
\end{equation*}
which proves \textit{(c)}.\\
\noindent \textit{(d)}: Fix $k \in \{0, \ldots, n-1\}$ and $t \in ]t_k^n, t_{k+1}^n[$. Then, $x_{k+1}^n \in \proj_{C(t_{k+1}^n)}^{\varepsilon_n, \eta_n}(\tau_n(t))$. Hence, by Lemma \ref{aprox-proj-II}, there exists $v_{k+1}^n \in C(t_{k+1}^n)$ such that $\Vert x_{k+1}^n - v_{k+1}^n \Vert < 2 \sqrt{\varepsilon_n^{\prime}} + \eta_n$ and 
$$\tau_n(t) - x_{k+1}^n \in \alpha_n(t) \partial_P d_{C(t_{k+1}^n)}(v_{k+1}^n) + (3\sqrt{\varepsilon_n^{\prime}} + \eta_n) \B, \quad \forall t \in ]t_k^n,t_{k+1}^n[,$$
where $\varepsilon_n^{\prime} = \varepsilon_n + 2\eta_n(d_{C(t_{k+1}^n)}(\tau_n(t)) + \sqrt{\varepsilon_n}) + \eta_n^{2}$ and $\alpha_n(t) = 4 \sqrt{\varepsilon_n^{\prime}} + d_{C(\theta_n(t))}(\tau_n(t))$. We observe that using (iii), we get $\varepsilon_n^{\prime} \leq  2 \varepsilon_n + 2K_3 \eta_n \mu_n + 4 \eta_n^{2} =: \sigma_n$ and $\sqrt{\sigma_n}/\mu_n \to 0$. By virtue of $(i)$,
$$\alpha_n(t) \leq 4 \sqrt{\sigma_n} + (L_C + h(x_n(\delta_n(t))) + \sqrt{\gamma})\mu_n + \eta_n =: \lambda_n(t) .$$
Then, for all $t \in ]t_k^n, t_{k+1}^n[$
$$- \mu_n(\dot{x}_n(t) - f(t, x_k^n)) \in \lambda_n(t) \partial_P d_{C(t_{k+1}^n)}(v_{k+1}^n) + (3 \sqrt{\sigma_n} + \eta_n) \B ,$$
which implies that for $t \in ]t_k^n, t_{k+1}^n[$
$$\dot{x_n}(t) \in - \frac{\lambda_n(t)}{\mu_n} \partial_{P}d_{C(\theta_n(t))}(v_{k+1}^{n}) + f(t, x_n(\delta_n(t))) + \frac{(3\sqrt{\sigma_n} + \eta_n)}{\mu_n} \B.$$

\end{proof}

\section{The Case of Prox-Regular Moving Sets}\label{Sec-4}

In this section, we prove the convergence of the inexact catching-up algorithm when the moving sets are uniformly prox-regular. Our results extend the convergence analysis carried out in the classical and inner approximate cases (see  \cite{MR2159846,MR4822735}).
\begin{theorem}
Suppose, in addition to the assumptions of Theorem \ref{scheme}, that $C(t)$ is $\rho$-uniformly prox-regular  for all $t\in [0,T]$, and for all $r > 0$, there exists a nonnegative integrable function $k_r$ such that for all $t\in [0,T]$ and $x_1, x_2\in r\B$ one has
\begin{equation}\label{monotonicity}
\langle v_1-v_2,x_1-x_2\rangle\leq k_r(t)\|x_1-x_2\|^2 \textrm{ for all } v_i \in F(t,x_i),  i=1,2.
\end{equation}
Then, the sequence of functions $(x_n)$ generated by the algorithm \eqref{piecewise_construction} and \eqref{approx_proj_step} converges uniformly to a Lipschitz continuous solution $x(\cdot)$ of \eqref{SP}. Moreover, if there exists $c \in L^{1}([0,T];\R_{+})$ such that
\begin{equation*}
\sup_{y \in F(t,x)} \Vert y \Vert \leq c(t)(\Vert x \Vert + 1) \textrm{ for all } x\in \H \textrm{ and  a.e. } t \in [0,T],
\end{equation*}
then the solution $x(\cdot)$ is unique.
\end{theorem}
\begin{proof}
 Consider $m,n\in\N$ with $m\geq n$ sufficiently large such that for all $t\in [0,T]$, $d_{C(\theta_n(t))}(x_m(t))<\rho$, this can be guaranteed by Theorem \ref{scheme}. Then, $\textrm{a.e.}$ $t\in [0,T]$ 
    \begin{equation*}
        \frac{\mathrm{d}}{\mathrm{d}t}\left(\frac{1}{2}\|x_n(t)-x_m(t)\|^2\right) = \langle \dot{x}_n(t)-\dot{x}_m(t), x_n(t)-x_m(t)\rangle.
    \end{equation*}
    Let $t\in [0,T]$ where the above equality holds. Let $k, j \in\{0,1,...,n-1\}$ such that $t\in ]t_k^n,t_{k+1}^n]$ and $t\in ]t_j^m,t_{j+1}^m]$. On the one hand,  we have that
    \begin{equation}\label{bound_1}
        \begin{aligned}
            \langle \dot{x}_n(t)-\dot{x}_m(t), x_n(t)-x_m(t)\rangle  &=  \ \langle \dot{x}_n(t)-\dot{x}_m(t),x_n(t)-x_{k+1}^n\rangle\\
            &+ \langle \dot{x}_n(t)-\dot{x}_m(t),x_{k+1}^n-v_{k+1}^n\rangle\\
            & + \langle \dot{x}_n(t)-\dot{x}_m(t),v_{k+1}^n -v_{j+1}^m\rangle\\
            &+\langle \dot{x}_n(t)-\dot{x}_m(t), v_{j+1}^m- x_{j+1}^m \rangle \\
            &+ \langle\dot{x}_n(t)-\dot{x}_m(t),x_{j+1}^m- x_m(t)\rangle\\
            &\leq    2K_7K_5(\mu_n +\mu_m) +4K_7(\sqrt{\varepsilon_n} + \sqrt{\varepsilon_m} + \sqrt{\sigma_n} + \sqrt{\sigma_n}) \\
         &+ 6K_7 (\eta_n + \eta_m)+\langle \dot{x}_n(t)-\dot{x}_m(t),v_{k+1}^n -v_{j+1}^m\rangle ,
        \end{aligned}
    \end{equation}
where $v_{k+1}^n\in C(t_{k+1}^n)$ and $v_{j+1}^m\in C(t_{j+1}^m)$ are the given in Theorem \ref{scheme}. We can see that 
\begin{equation*}
\begin{aligned}
\max\{d_{C(t_{k+1}^n)}(v_{j+1}^m),d_{C(t_{j+1}^m)}(v_{k+1}^n)\} &\leq \operatorname{Haus}(C(t_{j+1}^m),C(t_{k+1}^n))\\
&\leq L_C|t_{j+1}^m-t_{k+1}^n|\leq L_C(\mu_n+\mu_m).        
\end{aligned}
\end{equation*}
From now, $m,n\in\N$ are big enough such that $L_C(\mu_n+\mu_m)<\frac{\rho}{2}$. Moreover, as $h$ is $L_h$-Lipschitz, we have that for all $p\in\N$, $i\in \{0,1,...,p\}$ and $t\in[0,T]$
$$\|f(t,x_i^p)\|\leq h(x_i^p)+\sqrt{\gamma}\leq h(x_0) + L_hK_1+\sqrt{\gamma}=:\alpha.$$
From the other hand, using \eqref{xn-dot-inclusion} and Proposition \ref{prop-prox-regular} we have that 
\begin{equation*}
\begin{aligned}
        &\frac{1}{\Gamma}\max\{\left\langle \zeta_n-\dot{x}_n(t),v_{j+1}^m-v_{k+1}^n\right\rangle, \left\langle \zeta_m-\dot{x}_m(t),v_{k+1}^n-v_{j+1}^m\right\rangle\}\\
        \leq & \frac{2}{\rho}\|v_{k+1}^n-v_{j+1}^m\|^2+L_C(\mu_n+\mu_m) ,
\end{aligned}
\end{equation*}
where $\xi_n,\xi_m\in\mathbb{B}$, $\Gamma:=\sup\{\frac{\lambda_\ell(t)}{\mu_\ell}:t\in [0,T],\ell\in\N\}$ and $\zeta_i := f(t,x_i(\delta_i(t)))+\frac{3\sqrt{\sigma_i} + \eta_i}{\mu_i}\xi_i$ for $i\in \{n,m\}$. Therefore, we  have that
\begin{equation*}
\begin{aligned}
&\langle \dot{x}_n(t)-\dot{x}_m(t),v_{k+1}^n -v_{j+1}^m\rangle\\
&\hspace{17mm}=   \langle \dot{x}_n(t)-\zeta_n,v_{k+1}^n -v_{j+1}^m\rangle+\langle  \zeta_n-\zeta_m,v_{k+1}^n -v_{j+1}^m\rangle \\
&\hspace{17mm}+ \langle  \zeta_m-\dot x_m(t),v_{k+1}^n -v_{j+1}^m\rangle \\
&\hspace{17mm} \leq   2\Gamma(\frac{2}{\rho}\|v_{k+1}^n-v_{j+1}^m\|^2 + L_C(\mu_n+\mu_m))+\langle  \zeta_n-\zeta_m,v_{k+1}^n -v_{j+1}^m\rangle\\
&\hspace{17mm} \leq  \frac{4\Gamma}{\rho}(\|x_n(t)-x_m(t)\| + 2(\sqrt{\varepsilon_n}+\sqrt{\varepsilon_m} + \sqrt{\sigma_n} + \sqrt{\sigma_m}) + 3(\eta_n + \eta_m) + K_5(\mu_n+\mu_m))^2 \\
&\hspace{17mm}+ 2\Gamma L_C(\mu_n+\mu_m)
+\langle  \zeta_n-\zeta_m,v_{k+1}^n -v_{j+1}^m\rangle.
\end{aligned}
\end{equation*}
Moreover, using Theorem \ref{scheme} and property \eqref{monotonicity},
\begin{equation*}
\begin{aligned}
&  \langle  \zeta_n-\zeta_m,v_{k+1}^n -v_{j+1}^m\rangle\\
&\hspace{21mm} =   \langle  f(t,x_n(\delta_n(t)))-f(t,x_m(\delta_m(t))),x_n(\delta_n(t))-x_m(\delta_m(t))\rangle\\
&\hspace{21mm} + \langle  f(t,x_n(\delta_n(t)))-f(t,x_m(\delta_m(t))),v_{k+1}^n-x_{k+1}^n\rangle\\
&\hspace{21mm}+ \langle  f(t,x_n(\delta_n(t)))-f(t,x_m(\delta_m(t))),x_{k+1}^n-x_k^n\rangle\\
&\hspace{21mm} + \langle  f(t,x_n(\delta_n(t)))-f(t,x_m(\delta_m(t))),x_j^m-x_{j+1}^m\rangle \\
&\hspace{21mm}+ \langle  f(t,x_n(\delta_n(t)))-f(t,x_m(\delta_m(t))),x_{j+1}^m-v_{j+1}^m\rangle \\
&\hspace{21mm} + \frac{3\sqrt{\sigma_n} + \eta_n}{\mu_n}\langle\xi_n,v_{k+1}^n-v_{j+1}^m\rangle + \frac{3\sqrt{\sigma_m} + \eta_m}{\mu_m}\langle\xi_m,v_{j+1}^m-v_{k+1}^n\rangle\\
&\hspace{21mm}\leq  \ k(t) \Vert x_n(\delta_n(t))-x_m(\delta_m(t)) \Vert^2\\
&\hspace{21mm} + 2\alpha(2(\sqrt{\sigma_n} + \sqrt{\sigma_m}) + \sqrt{\varepsilon_n} + \sqrt{\varepsilon_m} + 2(\eta_n + \eta_m)+K_4(\mu_n+ \mu_m)) \\
&\hspace{21mm} + \frac{3\sqrt{\sigma_n} + \eta_n}{\mu_n}\|v_{k+1}^n-v_{j+1}^m\| + \frac{3\sqrt{\sigma_m} + \eta_m}{\mu_m}\|v_{j+1}^m-v_{k+1}^n\|\\
&\hspace{21mm}\leq   k(t)(\Vert x_n(t)-x_m(t)\Vert + 3(\sqrt{\varepsilon_n}+\sqrt{\varepsilon_m}) + 3(\eta_n + \eta_m)+(K_4+K_5)(\mu_n+\mu_m))^2\\
&\hspace{21mm} + 2\alpha(2(\sqrt{\sigma_n} + \sqrt{\sigma_m}) + \sqrt{\varepsilon_n} + \sqrt{\varepsilon_m} + 2(\eta_n + \eta_m)+K_4(\mu_n+ \mu_m)) \\
&\hspace{21mm} + \left(\frac{3\sqrt{\sigma_n} + \eta_n}{\mu_n}+\frac{3\sqrt{\sigma_m} + \eta_m}{\mu_m}\right)(2(\sqrt{\sigma_n} + \sqrt{\sigma_m}) + \eta_n + \eta_m +2K_2).
\end{aligned}
\end{equation*}
\noindent These two inequalities and \eqref{bound_1} yield 
\begin{equation*}
    \begin{aligned}
        &\frac{\mathrm{d}}{\mathrm{d}t}\Vert x_n(t)-x_m(t)\Vert^2\\
        &\hspace{10mm} \leq  4\left(\frac{4\Gamma}{\rho}+k(t)\right)\Vert x_n(t)-x_m(t)\Vert^2\\
        &\hspace{10mm} + 4K_7(K_5(\mu_n +\mu_m) +2(\sqrt{\varepsilon_n} + \sqrt{\varepsilon_m} + \sqrt{\sigma_n} + \sqrt{\sigma_n}) + 3 (\eta_n + \eta_m)) +4\Gamma L_C(\mu_n+\mu_m)\\
        &\hspace{10mm} + 4\alpha(2(\sqrt{\sigma_n} + \sqrt{\sigma_m}) + \sqrt{\varepsilon_n} + \sqrt{\varepsilon_m} + 2(\eta_n + \eta_m)+K_4(\mu_n+ \mu_m))\\
        &\hspace{10mm}+2\left(\frac{3\sqrt{\sigma_n} + \eta_n}{\mu_n}+\frac{3\sqrt{\sigma_m} + \eta_m}{\mu_m}\right)(2(\sqrt{\sigma_n}+\sqrt{\sigma_m}) + \eta_n + \eta_m + 2K_2)\\
        &\hspace{10mm}+\frac{16\Gamma}{\rho}(2(\sqrt{\varepsilon_n} + \sqrt{\varepsilon_m} + \sqrt{\sigma_n} + \sqrt{\sigma_m}) + 3(\eta_n + \eta_m) + K_5(\mu_n+\mu_m))^2\\
        &\hspace{10mm}+4k(t)\left(3(\sqrt{\varepsilon_n}+\sqrt{\varepsilon_m})+ 3(\eta_n + \eta_m) + (K_4+K_5)(\mu_n+\mu_m)\right)^2.
    \end{aligned}
\end{equation*} 
Hence, using Gronwall's inequality, we have for all $t\in [0,T]$ and $n,m$ big enough:
\begin{equation*}
    \Vert x_n(t)-x_m(t) \Vert^2\leq A_{m,n}\exp\left(\frac{16\Gamma}{\rho}T+4\int_0^T k(s)\mathrm{d}s\right),
\end{equation*}
where,
\begin{equation*}
    \begin{aligned}
        A_{m,n} &=\ 4\alpha T(2(\sqrt{\sigma_n} + \sqrt{\sigma_m}) + \sqrt{\varepsilon_n} + \sqrt{\varepsilon_m} + 2(\eta_n + \eta_m)+K_4(\mu_n+ \mu_m))\\
        & + 4TK_7(K_5(\mu_n +\mu_m) +2(\sqrt{\varepsilon_n} + \sqrt{\varepsilon_m} + \sqrt{\sigma_n} + \sqrt{\sigma_n}) + 3 (\eta_n + \eta_m)) \\
        &+4T\Gamma L_C(\mu_n+\mu_m)+2T\left(\frac{3\sqrt{\sigma_n} + \eta_n}{\mu_n}+\frac{3\sqrt{\sigma_m} + \eta_m}{\mu_m}\right)(2(\sqrt{\sigma_n}+\sqrt{\sigma_m}) + \eta_n + \eta_m + 2K_2)\\
        &+\frac{16T\Gamma}{\rho}(2(\sqrt{\varepsilon_n} + \sqrt{\varepsilon_m} + \sqrt{\sigma_n} + \sqrt{\sigma_m}) + 3(\eta_n + \eta_m) + K_5(\mu_n+\mu_m))^2\\
        &+4\Vert k \Vert_1(3(\sqrt{\varepsilon_n}+\sqrt{\varepsilon_m})+ 3(\eta_n + \eta_m) + (K_4+K_5)(\mu_n+\mu_m))^2.
    \end{aligned}
\end{equation*}
Since $A_{m,n}$ goes to $0$ when $m,n\to\infty$, it shows that $(x_n)$ is a Cauchy sequence in the space of continuous functions with the uniform convergence. Therefore, it converges uniformly to some continuous function $x\colon [0,T]\to \H$. It remains to check that $x$ is absolutely continuous, and it is the unique solution of \eqref{SP}. First of all, by Theorem \ref{scheme} and \cite[Lemma 2.2]{MR3626639}, $x$ is absolutely continuous and there is a subsequence of $(\dot{x}_{n})$ which converges weakly in $L^1([0,T];\H)$ to $\dot{x}$. So, without relabeling, we have $\dot{x}_n\rightharpoonup \dot{x}$ in $L^1([0,T];\H)$. On the other hand,  using Theorem \ref{scheme} and defining $v_n(t) := v_{k+1}^n$ for $t\in ]t_k^n,t_{k+1}^n]$ we have
\begin{equation*}
    \begin{aligned}
        \dot{x}_n(t) &\in -\frac{\lambda_n(t)}{\mu_n}\partial_P d_{C(\theta_n(t))}(v_n(t)) + f(t,x_n(\delta_n(t))) +\frac{3\sqrt{\sigma_n} + \eta_n}{\mu_n}\mathbb{B}\\
        &\in -\kappa_1\partial d_{C(\theta_n(t))}(v_n(t)) + \kappa_2\mathbb{B}\cap F(t,x_n(\delta_n(t))) + \frac{3\sqrt{\sigma_n} + \eta_n}{\mu_n}\mathbb{B},
    \end{aligned}
\end{equation*}
where, by Theorem \ref{scheme},  $\kappa_1$ and $\kappa_2$ are nonnegative numbers which do not depend of $n\in\N$ and $t\in [0,T]$. We also have $v_n\to x$, $\theta_n\to \operatorname{Id}_{[0,T]}$ and $\delta_n\to \operatorname{Id}_{[0,T]}$ uniformly. Theorem \ref{scheme} ensures that $x(t)\in C(t)$ for all $t\in [0,T]$. By Mazur's Lemma, there is a sequence $(y_j)$ such that for all $n$, $y_n\in \operatorname{co}(\dot{x}_k:k\geq n)$ and $(y_n)$ converges strongly to $\dot{x}$ in $L^1([0,T];\H)$. That is to say 
$$y_n(t)\in \operatorname{co}\left( -\kappa_1\partial d_{C(\theta_k(t))}(v_k(t)) + \kappa_2\mathbb{B}\cap F(t,x_k(\delta_k(t))) +\frac{3\sqrt{\sigma_k} + \eta_k}{\mu_k}\mathbb{B} :k\geq n\right).$$
Hence, there exists $(y_{n_j})$ which converges to $\dot{x}$ almost everywhere in $[0,T]$. Then, by virtue of \cite[Lemma~2]{MR4822735}, $\left(\H_3^F\right)$ and \cite[Lemma~3]{MR4822735}, we obtain that $$\dot{x}(t)\in -\kappa_1\partial d_{C(t)}(x(t)) +\kappa_2\mathbb{B}\cap F(t,x(t)) \textrm{ for a.e. } t\in [0,T].$$
Since $\partial d_{C(t)}(x(t))\subset N(C(t);x(t))$ for all $t\in[0,T]$, we obtain that $x$ is the solution of \eqref{SP}. Finally, the uniqueness of solutions can be derived following an argument similar to that of \cite[Theorem~2]{MR4822735}.
\end{proof}
\begin{remark}
It is worth emphasizing that property \eqref{monotonicity} is a classical monotonicity assumption in the theory of the existence of solutions for differential inclusions (see, e.g., \cite[Theorem 10.5]{MR1189795}).
\end{remark}
\begin{remark}[Rate of convergence]
In the precedent proof, we have established the following estimation:
$$\Vert x_n(t) - x_m(t) \Vert^{2} \leq A_{m,n} \exp \left(\frac{16 \Gamma}{\rho}T + 4 \int_{0}^{T} k(s) \mathrm{d}s \right)$$
for $m,n$ such that $\mu_n + \mu_m \leq \frac{\rho}{2 L_C}$. Hence, by letting $m \to + \infty$, we obtain that 
$$
\Vert x_n(t) - x(t) \Vert^{2} \leq A_n \exp \left(\frac{16 \Gamma}{\rho}T + 4 \int_{0}^{T} k(s) \mathrm{d}s \right) \text{ for all }n > \frac{2 L_C T}{\rho},
$$ 
where
$$
A_n:= \lim_{m \to \infty} A_{m,n} \leq D \left(\sqrt{\varepsilon_n} + \eta_n + \mu_n + \frac{\sqrt{\varepsilon_n}}{\mu_n} + \frac{\eta_n}{\mu_n} + \sqrt{\eta_n \mu_n}\right) ,
$$
where $D>0$. Hence, the above estimation provides a rate of convergence for our scheme.
\end{remark}

\section{The Case of Uniformly Subsmooth Moving Sets}\label{Sec-5}

In this section, we prove the convergence of the inexact catching-up algorithm when the moving sets are uniformly subsmooth. Our results extend the convergence analysis carried out  \cite{MR3574145,MR4822735}. In contrast to the prox-regular case, uniform subsmoothness fails to guarantee a sufficiently strong monotonicity of the normal cone necessary to ensure the existence and uniqueness of solutions. Consequently, in the subsequent analysis, we shall assume that the moving sets are ball-compact.
\begin{theorem}
Suppose, in addition to assumptions of Theorem \ref{scheme}, that the family $(C(t))_{t\in [0,T]}$ is equi-uniformly subsmooth and the sets $C(t)$ are ball-compact for all $t\in [0,T]$. Then, the sequence of continuous functions $(x_n)$ generated by algorithm \eqref{piecewise_construction} and \eqref{approx_proj_step} converges uniformly (up to a subsequence) to a Lipschitz continuous solution $x(\cdot)$ of \eqref{SP}.
\end{theorem}

\begin{proof}
    From Theorem \ref{scheme} we have for all $n\in\N$ and $k\in \{0,\ldots,n-1\}$, there is $v_{k+1}^n\in C(t_{k+1}^n)$ such that $\|v_{k+1}^n-x_{k+1}^n\|<2\sqrt{\sigma_n} + \eta_n$ and for all $t\in ]t_k^n,t_{k+1}^n]$:
\begin{equation*}
    \dot{x}_n(t)\in -\frac{\lambda_n(t)}{\mu_n}\partial_P d_{C(\theta_n(t))}(v_{k+1}^n) + f(t,x_n(\delta_n(t))) +\frac{3\sqrt{\sigma_n} + \eta_n}{\mu_n}\mathbb{B}.
\end{equation*}
where $\lambda_n(t) = 4\sqrt{\sigma_n} + (L_C + h(x(\delta_n(t)))+\sqrt{\gamma})\mu_n + \eta_n$. We define $\nu := \sup_{n \in \N} \frac{4\sqrt{\sigma_n} + \eta_n}{\mu_n}$. As $h$ is $L_h$-Lipschitz it follows that $$\lambda_n(t)\leq (\nu+L_C+h(x_0)+\sqrt{\gamma} + L_hK_1)\mu_n.$$ Defining $v_n(t):=v_{k+1}^n$ on $]t_k^n,t_{k+1}^n]$, then for all $n\in\N$ and almost all $t\in [0,T]$
\begin{equation*}
    \begin{aligned}
        \dot{x}_n(t) &\in -M\partial_P d_{C(\theta_n(t))}(v_n(t)) + f(t,x_n(\delta_n(t))) +\frac{3\sqrt{\sigma_n} + \eta_n}{\mu_n}\mathbb{B}\\
        &\in -M\partial d_{C(\theta_n(t))}(v_n(t)) + M\mathbb{B}\cap F(t,x_n(\delta_n(t))) +\frac{3\sqrt{\sigma_n} + \eta_n}{\mu_n}\mathbb{B}.     
    \end{aligned}
\end{equation*}
where $M:= \nu+L_C+h(x_0) + L_hK_1+\sqrt{\gamma}$. Moreover, by Theorem \ref{scheme},  we have for all $t\in [0,T]$
\begin{equation}\label{boundsubsmooth1}
\begin{aligned}
d_{C(t)}(x_n(t))\leq d_{C(\theta_n(t))}(x_n(t)) + L_C\mu_n\leq   (K_6+2L_C)\mu_n+2\sqrt{\varepsilon_n} + 3\eta_n.
\end{aligned}
\end{equation}
Next, fix $t\in [0,T]$ and define $K(t):=\{x_n(t) : n\in \N\}$. We claim that $K(t)$ is relatively compact. Indeed, let $x_m(t) \in K(t)$ and take $y_m(t)\in \Proj_{C(t)}(x_m(t))$ (the projection exists due to the ball compactness of $C(t)$ and the boundedness of $K(t)$). Moreover, according to \eqref{boundsubsmooth1} and Theorem \ref{scheme},
\begin{equation*}
\begin{aligned}
\Vert y_n(t)\Vert & \leq d_{C(t)}(x_n(t)) + \Vert x_n(t)\Vert \leq  (K_6 + 2L_C)\mu_n + 2\sqrt{\varepsilon_n}+ 3 \eta_n + K_2.
\end{aligned}
\end{equation*}
This entails that $y_n(t)\in C(t)\cap R\, \mathbb{B}$ for all $n\in\mathbb{N}$ for some $R>0$. Thus, by the ball compactness of $C(t)$, there exists a subsequence $(y_{m_k}(t))$ of $(y_m(t))$ converging to some $y(t)$ as $k\to +\infty$. Then, 
\begin{equation*}
\begin{aligned}
\Vert x_{m_k}(t)-y(t)\Vert &\leq d_{C(t)}(x_{m_k}(t))+\Vert y_{m_k}(t)-y(t)\Vert \\
&\leq (K_6 + 2L_C)\mu_{m_k} + 2\sqrt{\varepsilon_{m_k}} + 3\eta_n + \Vert y_{m_k}(t)-y(t) \Vert,
\end{aligned}
\end{equation*}
which implies that $K(t)$ is relatively compact. Moreover, by Theorem \ref{scheme} that $K:=(x_n)$ is equicontinuous.  Therefore, by virtue of Theorem \ref{scheme}, Arzela-Ascoli's and \cite[Lemma~2.2]{MR3626639}, we obtain the existence of a Lipschitz function $x(\cdot)$ and a subsequence $(x_j)$ of $(x_n)$ such that
\begin{enumerate}[label=(\roman{*})]
\item $(x_j)$ converges uniformly to $x$ on $[0,T]$.
\item  $\dot{x}_j\rightharpoonup \dot{x}$ in $L^1\left([0,T];\H\right)$.
\item $x_j(\theta_j(t))\to x(t)$ for all $t\in [0,T]$.
\item $x_j(\delta_j(t))\to x(t)$ for all $t\in [0,T]$.
\item $v_j(t)\to x(t)$ for all $t\in [0,T]$.
\end{enumerate}
From \eqref{boundsubsmooth1} it is clear that $x(t)\in C(t)$ for all $t\in [0,T]$. By Mazur's Lemma, there is a sequence $(y_j)$ such that for all $j$, $y_j\in \text{co}(\dot{x}_k:k\geq j)$ and $(y_j)$ converges strongly to $\dot{x}$ in $L^1([0,T];\H)$. That is $$y_j(t)\in \operatorname{co}\left( -M\partial d_{C(\theta_n(t))}(v_n(t)) + M\mathbb{B}\cap F(t,x_n(\delta_n(t))) +\frac{3\sqrt{\sigma_n} + \eta_n}{\mu_n}\mathbb{B} :n\geq j\right).$$ On the other hand, there exists $(y_{n_j})$ which converges to $\dot{x}$ almost everywhere in $[0,T]$. Then, using \cite[Lemma~2]{MR4822735}, \cite[Lemma~3]{MR4822735} and $(\H_{1}^{F})$, we have $$\dot{x}(t)\in -M\partial d_{C(t)}(x(t)) +M\mathbb{B}\cap F(t,x(t))  \textrm{ a.e. }$$
Finally, since $\partial d_{C(t)}(x(t))\subset N(C(t);x(t))$ for all $t\in [0,T]$, it follows that $x$ solves \eqref{SP}.
\end{proof}

\section{The Case of a Fixed Set}\label{Sec-6}
In this section, we prove the convergence of the inexact catching-up algorithm for the sweeping process driven by a fixed set:
\begin{equation}\label{fixed-set-SP}
\left\{
\begin{aligned}
\dot{x}(t)&\in -N\left(C;x(t)\right)+F(t,x(t)) & \textrm{ a.e. } t\in [0,T],\\
x(0)&=x_0\in C,
\end{aligned}
\right.
\end{equation}
where $C\subset \H$ and $F\colon [0,T]\times \H \rightrightarrows \H$ is a set-valued map defined as above. It is worth emphasizing that the above dynamical system is strongly related to the concept of a projected dynamical system (see, e.g., \cite{MR2185604}). Our results extend the convergence analysis carried out in the classical and inner approximate cases (see  \cite{MR3956966,MR4822735}). It is worth to emphasizing that, in this case, no regularity of the set $C$ is required.
\begin{theorem}
Let $C\subset \H$ be a ball-compact set and $F\colon [0,T]\times \H\rightrightarrows \H$ be a set-valued map satisfying $(\H_1^F), (\H_1^2)$ and $(\H_3^F)$. Then, the sequence of functions $(x_n)$ generated by the algorithm \eqref{approx_proj_step} converges uniformly (up to a subsequence) to a  Lipschitz solution  $x(\cdot)$ of  \eqref{fixed-set-SP} such that
\begin{equation*}
\begin{aligned}
\Vert \dot{x}(t)\Vert  & \leq  2(h(x(t))+\sqrt{\gamma}) & \textrm{ a.e. } t\in [0,T].
\end{aligned}
\end{equation*}

\end{theorem}
\begin{proof} We are going to use the properties of Theorem \ref{scheme}, where now we have $L_C = 0$. First of all, from Theorem \ref{scheme} we have for all $n\in\N$ and $k\in \{0,1,\ldots,n-1\}$, there is $v_{k+1}^n\in C$ such that $\|v_{k+1}^n-x_{k+1}^n\|<2\sqrt{\sigma_n} + \eta_n$ and for all $t\in ]t_k^n,t_{k+1}^n]$:
\begin{equation*}
    \dot{x}_n(t)\in -\frac{\lambda_n(t)}{\mu_n}\partial_P d_{C}(v_{k+1}^n) + f(t,x_n(\delta_n(t))) +\frac{3\sqrt{\sigma_n}+\eta_n}{\mu_n}\mathbb{B},
\end{equation*}
where $\lambda_n(t) = 4\sqrt{\sigma_n} + \eta_n + (h(x(\delta_n(t)))+\sqrt{\gamma})\mu_n$. Defining $v_n(t):=v_{k+1}^n$ on $]t_k^n,t_{k+1}^n]$, we get that for all $n\in\N$ and \textrm{a.e.} $t\in [0,T]$
\begin{equation*}
    \begin{aligned}
        \dot{x}_n(t) &\in  -\frac{\lambda_n(t)}{\mu_n}\partial_P d_{C}(v_n(t)) + f(t,x_n(\delta_n(t))) +\frac{3\sqrt{\sigma_n}+\eta_n}{\mu_n}\mathbb{B}\\
        &\in  -\frac{\lambda_n(t)}{\mu_n}\partial d_{C}(v_n(t))+ (h(t,x_n(\delta_n(t)))+\sqrt{\gamma})\mathbb{B}\cap F(t,x_n(\delta_n(t))) +\frac{3\sqrt{\sigma_n}+\eta_n}{\mu_n}\mathbb{B}.   
    \end{aligned}
\end{equation*}
Moreover, by Theorem \ref{scheme},  we have 
\begin{equation*}
d_C(x_n(t))\leq  K_6\mu_n+2\sqrt{\varepsilon_n} + 3 \eta_n \textrm{ for all } t\in [0,T].
\end{equation*}
Next, fix $t\in [0,T]$ and define $K(t):=\{x_n(t) : n\in \mathbb{N}\}$. We claim that $K(t)$ is relatively compact. Indeed, let $x_m(t) \in K(t)$ and take $y_m(t)\in \operatorname{Proj}_C(x_m(t))$ (the projection exists due to the ball compactness of $C$ and the boundedness of $K(t)$). Moreover, according to the above inequality and Theorem \ref{scheme},
\begin{equation*}
\Vert y_n(t)\Vert  \leq d_C(x_n(t))+\Vert x_n(t)\Vert \leq   K_6\mu_n+2\sqrt{\varepsilon_n} + 3\eta_n + K_2,
\end{equation*}
which entails that $y_n(t)\in C\cap R\, \mathbb{B}$ for all $n\in\mathbb{N}$ for some $R>0$. Thus, by the ball-compactness of $C$, there exists a subsequence $(y_{m_k}(t))$ of $(y_m(t))$ converging to some $y(t)$ as $k\to +\infty$. Then, 
\begin{equation*}
\begin{aligned}
\Vert x_{m_k}(t)-y(t)\Vert &\leq d_{C}(x_{m_k}(t))+\Vert y_{m_k}(t)-y(t)\Vert \\
&\leq K_6\mu_{m_k}+2\sqrt{\varepsilon_{m_k}}+ 3\eta_{m_k} + \Vert y_{m_k}(t)-y(t)\Vert,
\end{aligned}
\end{equation*}
which implies that $K(t)$ is relatively compact. Moreover, by Theorem \ref{scheme}, the set $K:=(x_n)$ is equicontinuous.  Therefore, by virtue of Theorem \ref{scheme}, Arzela-Ascoli's and \cite[Lemma~2.2]{MR3626639}, we obtain the existence of a Lipschitz function $x$ and a subsequence $(x_j)$ of $(x_n)$ such that
\begin{enumerate}[(i)]
\item $(x_j)$ converges uniformly to $x$ on $[0,T]$.
\item  $\dot{x}_j\rightharpoonup \dot{x}$ in $L^1\left([0,T];\H\right)$.
\item $x_j(\theta_j(t))\to x(t)$ for all $t\in [0,T]$.
\item $x_j(\delta_j(t))\to x(t)$ for all $t\in [0,T]$.
\item $v_j(t)\to x(t)$ for all $t\in [0,T]$.
\item $x(t)\in C$  for all $t\in [0,T]$.
\end{enumerate} 
By Mazur's Lemma, there is a sequence $(y_j)$ such that for all $j$, $y_j\in \text{co}(\dot{x}_k:k\geq j)$ and $(y_j)$ converges strongly to $\dot{x}$ in $L^1([0,T];\H)$. i.e., 
$$y_j(t)\in \text{co}( -\alpha_n\partial d_{C}(v_n(t)) + \beta_n\mathbb{B}\cap F(t,x_n(\delta_n(t))) +\frac{3\sqrt{\sigma_n} + \eta_n}{\mu_n}\mathbb{B} :n\geq j),
$$ 
where $\alpha_n:= \frac{4\sqrt{\sigma_n} + \eta_n}{\mu_n}+h(x_n(\delta_n(t)))+\sqrt{\gamma}$ and $\beta_n:= h(x_n(\delta_n(t))) + \sqrt{\gamma}$. On the other hand, there exists $(y_{n_j})$ which converges to $\dot{x}$ a.e. in $[0,T]$. Then, using \cite[Lemma~2]{MR4822735}, \cite[Lemma~3]{MR4822735} and $(\H_1^F)$, we have 
$$\dot{x}(t)\in -(h(x(t))+\sqrt{\gamma})\partial d_{C}(x(t)) +(h(x(t))+\sqrt{\gamma})\mathbb{B}\cap F(t,x(t))  \textrm{ for a.e. } t\in [0,T].
$$
Finally, since $\partial d_{C}(x(t))\subset N(C;x(t))$ for all $t\in [0,T]$, we obtain that  $x$ solves \eqref{fixed-set-SP}.
\end{proof}

\section{An Application to Complementarity Dynamical Systems}\label{Sec-7}

In this section, we will apply our enhanced algorithm to complementarity dynamical systems. These systems have garnered increasing attention due to their applications in fields such as mechanics, economics, transportation, control systems and electrical circuits, see e.g., \cite{MR2185604, MR2731287, MR2049810}.  Complementarity dynamical systems combine ordinary differential equations with complementarity conditions, which can, in turn, be equivalently expressed using variational inequalities or specific differential inclusions, see e.g., \cite{MR3467591, Acary-Brogliato-2008}.

Following \cite{MR2731287}, let us consider the following class of linear complementarity dynamical systems
\begin{equation}\label{CDS}
\begin{cases}
    \dot{x}(t) = Ax(t) + B \zeta(t) + E u(t) \\
    0 \leq \zeta(t) \perp w(t) = Cx(t) + D\zeta(t) + Gu(t) + F \geq 0,
\end{cases}
\end{equation}
where the matrices and vectors $A,B,C,D,E,F,G$ are constant of suitable dimensions, $x(t) \in \R^{n}$, $u(t) \in \R^{p}$, $\zeta(t) \in \R^{m}$. We consider the special case where $D=0$ and assume the existence of a symmetric, positive-definite matrix $P$ such that $PB = C^{\top}$. It was shown in \cite{MR2731287} that by defining $R=\sqrt{P}$ and introducing the change of variables $z(t) = Rx(t)$, the system \eqref{CDS} can be reformulated as the following perturbed sweeping process:
\begin{equation*}
\dot{z}(t) \in -N(S(t);z(t)) + RAR^{-1}z(t) + REu(t) ,
\end{equation*}
where  $S(t) := R(K(t)) = \{Rx : \ x \in K(t)\}$ and $K(t)$ is the closed convex polyhedral set $$K(t):= \{x \in \R^{n} \ | \ Cx + Gu(t) + F \geq 0\} .$$
Fix $x\in \H$ and $\varepsilon, \eta>0$. To apply the inexact catching-up algorithm, we must devise a numerical method to find $\varepsilon-\eta$ approximate projections. Since obtaining the projection involves a quadratically constrained problem, we will use the primal-dual approach (see, e.g., \cite{MR4621666}) to the (primal) optimization problem:
\begin{equation}\label{P_t_canon}
\begin{aligned}
d_{S(t)}^2(x)=\inf_{y\in K(t)}\Vert x-Ry\Vert^2 = \inf_{Ay\leq b}  y^{\top} P y + 2 f^{\top} y
\end{aligned}
\end{equation}
where $Q:= P$, $f:= -Rx$, $A:= -C$ and $b:= Gu(t) + F$. The dual formulation of \eqref{P_t_canon} is 
\begin{equation}\label{D_t_canon}
\begin{aligned}
\max_{\lambda \in \R^{m}_+} & -\lambda^{\top} A Q^{-1} A^{\top} \lambda -2(AQ^{-1}f + b)^{\top} \lambda -f^{\top} Q^{-1}f.
\end{aligned}
\end{equation}
Moreover, the primal and dual problems are linked through the relation:
\begin{equation}\label{KKT}
y^{*} = -Q^{-1}(f + A^{\top} \lambda^{*}) ,
\end{equation}
where $y^{*}$ and $\lambda^{*}$ are the primal and dual solutions, respectively. Hence, we can solve the dual problem using the projected gradient descent method:
\begin{equation}\label{projected-gradient}
\lambda_{k+1} = \left[ \lambda_k - \frac{1}{\lambda_{\textrm{max}}(CB)}\left(CB \lambda_k + B^{\top}Rx + Gu(t) + F\right)\right]_{+} ,
\end{equation}
where $[\cdot]_{+}$ denotes the projection onto the nonnegative orthant (see, e.g.,  \cite[Lemma~6.26]{MR3719240}). Finally, the primal solution can be recovered through relation \eqref{KKT}. 
\begin{remark}
Here, the contribution of our inexact method can be clearly observed. It is easy to see that the proposed algorithm for calculating the projection does not necessarily yield points that remain within the set, highlighting the importance of approaching the projection from any point. 
\end{remark}
The next result provides some properties of the proposed numerical method.
\begin{lemma}\label{lema-dual-bound}
Let $y^{*}$ and ${\lambda^{*}}$ be solutions of \eqref{P_t_canon} and \eqref{D_t_canon}, respectively. Let $(\lambda_k)$ be the sequence generated by \eqref{projected-gradient}. Define $y_k = B\lambda_k + R^{-1}x$ for all $k \in \N$. Then, the following assertions hold:
\begin{itemize}
\item[(i)] For all $k\in \mathbb{N}$, $\Vert y_k - y^{*} \Vert \leq \Vert B \Vert \Vert \lambda_k - \lambda^{*} \Vert$.
\item[(ii)] Let $\varepsilon, \eta>0$  and suppose that, for some $\bar{k} \in \N$, the following condition hold:
$$\Vert \lambda_{\bar{k}} - \lambda^{*} \Vert \leq \operatorname{max}\{ \frac{\varepsilon}{M}, \frac{\eta}{\Vert B \Vert \Vert R \Vert}\} \textrm{ with } M:= \sup_{k \in \N} \left(\Vert P \Vert \Vert B(\lambda_k + \lambda^{*})+ 2R^{-1}x\Vert + 2\Vert Rx\Vert\right) \Vert B \Vert .
$$
Then, $z:= R y_{\bar{k}}\in \proj_{S(t)}^{\varepsilon, \eta}(x)$.
\end{itemize}
\end{lemma}
\begin{proof} Assertion $(i)$ follows directly from relation \eqref{KKT}. To prove $(ii)$, we observe that
\begin{equation}\label{bound-eps}
\begin{aligned}
    \Vert x - Ry_k \Vert^{2} - \Vert x - Ry^{*} \Vert^{2} & = y_k^{\top} P y_k - y^{*\top} P y^{*} - 2(Rx)^{\top}(y_k - y^{*}) \\
    &=  (y_k + y^{*})^{\top}P(y_k-y^{*}) + 2(-Rx)^{\top}(y_k - y^{*})\\
    &\leq  \left(\Vert P \Vert \Vert y_k + y^{*}\Vert + 2\Vert Rx\Vert\right)\Vert y_k - y^{*}\Vert\\
    &\leq  \left(\Vert P \Vert \Vert B(\lambda_k + \lambda^{*})+ 2R^{-1}x\Vert + 2\Vert Rx\Vert\right) \Vert B \Vert \Vert \lambda_k - \lambda^{*}\Vert ,
\end{aligned}
\end{equation}
where we have used (i).  Since the dual problem is a strictly convex quadratic program, $(\lambda_k)$ converges to the unique solution $\lambda^{*}$, and therefore $M<+\infty$. Hence, by using that $\Vert \lambda_k - \lambda^{*} \Vert \leq \varepsilon/M$, we obtain that $\Vert x - Ry_k \Vert^{2} \leq d_{S(t)}^2(x) + \varepsilon$.  
Moreover, since $\Vert \lambda_k - \lambda^{*} \Vert \leq  \frac{\eta}{\Vert B \Vert \Vert R \Vert}$, it follows from (i) that 
$$\Vert \bar{y} - y^{*} \Vert \leq \frac{\eta}{\Vert R \Vert},$$
which implies that $z \in \proj_{S(t)}^{\varepsilon, \eta}(x)$.
\end{proof}
It is worth mentioning that the number of iterations required to achieve a certain precision (and ensure an $\varepsilon-\eta$ approximate projection) can be estimated using classical results from the projected gradient method (see, e.g., \cite[Theorem 2.2.8]{MR2142598}). We end this section by applying our numerical method to a problem involving electrical circuits with ideal diodes. The example was considered previously in \cite[Example 2.52]{Acary-Brogliato-2008}.
\begin{example}
Let us consider the electrical circuit with ideal diodes shown in Figure \ref{circuito-1}.
\begin{figure}[h]
    \centering
            \includegraphics[width=0.45\textwidth]{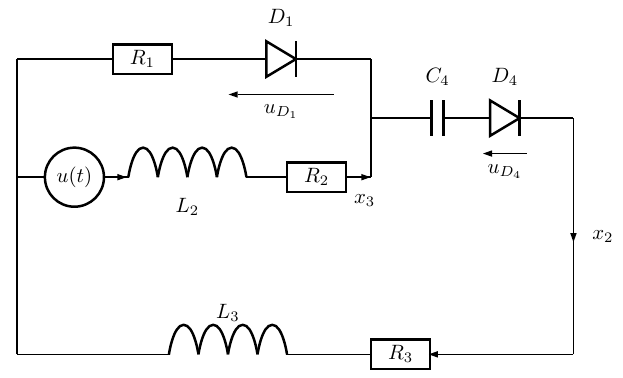} 
        \caption{Electrical circuit with ideal diodes.}
        \label{circuito-1}
\end{figure}
\\
Here  $R_1, R_2, R_3 \geq 0$, $L_2, L_3 > 0$, $C_4 > 0$. The presence of the diodes generates the complementarity relationships  $0 \leq -u_{D_4} \perp x_2 \geq 0$ and $0 \leq -u_{D_1} \perp -x_3 + x_2 \geq 0$, where $u_{D_4}$ and $u_{D_1}$ are the voltages of the diodes. The dynamics of the circuit are given by the following system:
\begin{equation}\label{Electrical-Circuit}
\begin{cases}
    \dot{x_1}(t) = x_2(t) \\
    \dot{x_2}(t) = - \left( \frac{R_1 + R_3}{L_3}\right)x_2(t) + \frac{R_1}{L_3} x_3(t) - \frac{1}{L_3 C_4} x_1(t) + \frac{1}{L_3}\zeta_1(t) +  \frac{1}{L_3}\zeta_2(t) + \frac{u(t)}{L_3}, \\
    \dot{x_3}(t) = - \left( \frac{R_1 + R_2}{L_2}\right)x_3(t) + \frac{R_1}{L_2}x_2(t) - \frac{1}{L_2}\zeta_1(t) + \frac{u(t)}{L_2}, \\
    0 \leq \begin{pmatrix}
        \zeta_1(t) \\ 
        \zeta_2(t)
    \end{pmatrix} \perp \begin{pmatrix}
        -x_3(t) + x_2(t) \\
        x_2(t)
    \end{pmatrix} \geq 0,
\end{cases}
\end{equation}
where $x_1(\cdot)$ is the time integral of the current across the capacitor, $x_2(\cdot)$ is the current across the capacitor, and $x_3(\cdot)$ is the current across the inductor $L_2$ and resistor $R_2$, $- \zeta_1$ is the voltage of the diode $D_1$ and $-\zeta_2$ is the voltage of the diode $D_4$. The system in \eqref{Electrical-Circuit} can be written compactly as 
\begin{equation*}
    \begin{cases}
        \dot{x}(t) = Ax(t) + B\zeta(t) + Eu(t)\\
        0 \leq \zeta(t) \perp y(t) = Cx(t) \geq 0 ,
    \end{cases}
\end{equation*}
with 
\begin{equation*}
    A = \begin{pmatrix}
        0 & 1 & 0 \\
        -\frac{1}{L_3C_4} & - \frac{R_1 + R_3}{L_3} & \frac{R_1}{L_3}\\
        0 & \frac{R_1}{L_2} & - \frac{R_1 + R_2}{L_2}
    \end{pmatrix}, B = \begin{pmatrix}
    0 & 0 \\
    \frac{1}{L_3} & \frac{1}{L_3} \\
    -\frac{1}{L_2} & 0
\end{pmatrix}, C=
\begin{pmatrix}
     0 & 1 & -1 \\
     0 & 1 & 0
\end{pmatrix}, E = 
\begin{pmatrix}
    0 \\
    \frac{1}{L_3} \\
    \frac{1}{L_2}
\end{pmatrix},
\end{equation*}
which is a particular case of $\eqref{CDS}$ with $D = 0, F = 0$ and $G = 0$. Moreover, $PB = C^T$ holds with 
$$P = \begin{pmatrix}
        1 & 0 & 0 \\
        0 &  L_3 & 0\\
        0 & 0 & L_2
    \end{pmatrix} \implies R = \begin{pmatrix}
        1 & 0 & 0 \\
        0 &  \sqrt{L_3} & 0\\
        0 & 0 & \sqrt{L_2} 
    \end{pmatrix}.
    $$
To apply the inexact catching-up algorithm, we consider $n = 100$, a uniform partition  $\left(t_k^n\right)_{k = 0}^{n}$ of $[0, 1]$ with  $\mu_n = \frac{1}{n}$, $\varepsilon_n = \frac{1}{n^{2.1}}, \eta_n = \frac{1}{n^{1.05}}$. As discussed earlier, the variable $z = Rx$ satisfies: 
$$
\dot{z}(t) \in -N(S;z(t)) + f(t, z(t)),
$$
where $f(t,x) = RAR^{-1}x + REu(t)$, $S = \{Rx : x \in K\}$  and $K := \{x \in \R^{3}: -Cx \leq 0\}$. We apply the inexact catching-up algorithm by computing for each $k \in \{0, \ldots, n-1\}$
\begin{equation}\label{Catching-up}
z_{k+1}^n \in \proj_{S}^{\varepsilon_n, \eta_n}(z_k^n+RAR^{-1}z_k^n \mu_n + RE \int_{t_k^n}^{t_{k+1}^n}  u(s) \mathrm{d}s).
\end{equation}
Hence, we consider for each $k \in \{0,1, \ldots, n-1\}$ the associated dual problem
\begin{equation*}
\begin{aligned}
\min_{\lambda \in \R^{m}_+}  \lambda^{\top} CB \lambda + 2\left(B^{\top} Rw_{k}^{n}\right)^{\top}\lambda,
\end{aligned}
\end{equation*}
where $w_k^n:= z_k^n+RAR^{-1}z_k^n \mu_n + RE \int_{t_k^n}^{t_{k+1}^n}  u(s) \mathrm{d}s$ (the integral is evaluated through a classical integration technique). We apply the projected gradient descent (see algorithm \eqref{projected-gradient})
\begin{equation*}
\lambda_{j+1} = \left[ \lambda_j - \frac{1}{\lambda_{\max}(CB)}\left(CB \lambda_j + B^{\top}Rw_{k}^{n}\right)\right]_{+} ,
\end{equation*}
to get an approximate dual solution. Then, from \eqref{KKT}, we obtain an approximate primal solution $y_k$. Finally, $z_{k+1}:=Ry_k$ satisfies \eqref{Catching-up}. 

Figure \ref{Figurex1-x3} shows the numerical result for $R_1 = 1, R_2 = 2, R_3 = 1$, $L_2 = 1, L_3 = 2, C_4 = 1$, $u(t) = 16\sin(6\pi t) - 0.5$ and initial condition $x_0 = (0,0,0)$.

\begin{figure}[htbp]
\begin{center}$
\begin{array}{cc}
\includegraphics[scale=0.7]{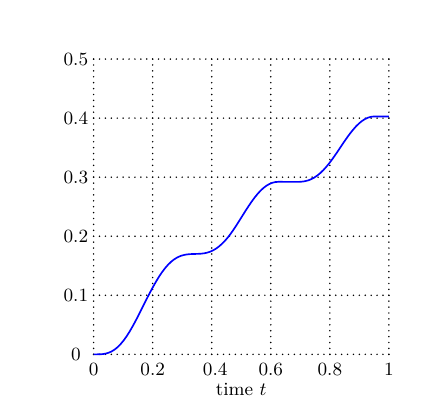}&
\includegraphics[scale=0.7]{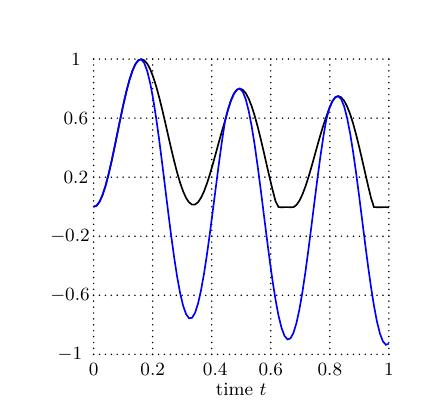}
\end{array}$
\end{center}
\caption{On the left solution $x_1$ and on the right solutions $x_2$ (black) and $x_3$ (blue) for $R_1 = 1, R_2 = 2, R_3 = 1$, $L_2 = 1, L_3 = 2, C_4 = 1$, $u(t) = 16\sin(6\pi t) - 0.5$ and $x_0 = (0,0,0)$.}
\label{Figurex1-x3}
\end{figure}
\end{example}
It is worth noting that the above example is a particular case of \eqref{CDS} with $G = 0$. The case where $G \neq 0$ is especially interesting, as it causes the set $K(t):=\{x \in\R^{n} : Cx+Gu(t)+F \geq 0\}$ to vary over time. This falls within the scope of our results as long as $u(t)$ is Lipschitz. However, if $u(t)$ is discontinuous, then $K(t)$ will also be discontinuous, and solutions will be  discontinuous as well. Although the discontinuous case is not addressed by the developments of this work, the inexact catching algorithm seems to be effective, as is shown in Figure \ref{Discontinuous-case}.
\begin{figure}[htbp]
\begin{center}$
\begin{array}{cc}
\includegraphics[scale=0.7]{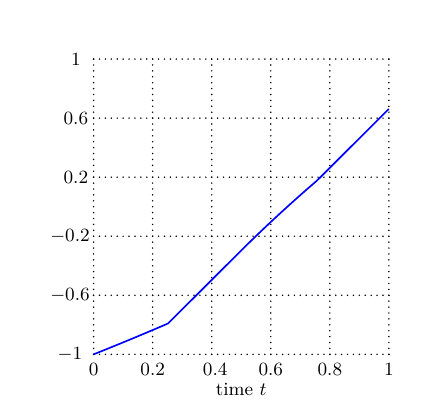}&
\includegraphics[scale=0.7]{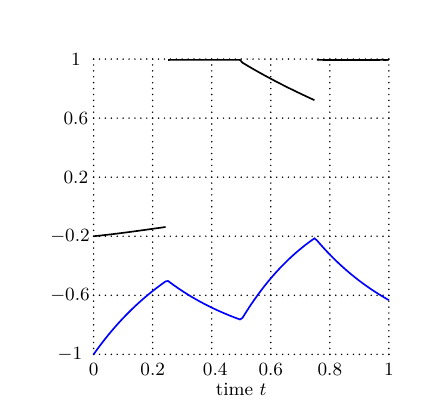}
\end{array}$
\end{center}
\caption{On the left solution $x_1$ and on the right solutions $x_2$ (black) and $x_3$ (blue) for $R_1 = 1, R_2 = 2, R_3 = 1$, $L_2 = 1, L_3 = 2, C_4 = 1$,  $G=(0,1)^t$, $u(t) = \operatorname{sign}(\sin(4 \pi t))$ and $x_0 = (0,0,0)$.}
\label{Discontinuous-case}
\end{figure}
\newpage
\section{Concluding Remarks}

In this paper, we present an inexact version of the catching-up algorithm for sweeping processes.  Building on the work in \cite{MR2731287}, we define a new notion of approximate projection, called $\varepsilon-\eta$ approximate projection, which is compatible with any numerical method for approximating exact projections, as this new notion is not restricted to remain strictly within the set. We provide several properties of $\varepsilon-\eta$ approximate projections, which enable us to prove the convergence of the inexact catching-up algorithm in three general frameworks: prox-regular moving sets, subsmooth moving sets, and merely closed sets.  

Additionally, we apply our numerical results to address complementarity dynamical systems, particularly electrical circuits with ideal diodes. In this context, we implement the inexact catching-up algorithm using a primal-dual  method, which typically does not guarantee a feasible point. 

 Future research could focus on extending the results of this paper to encompass additional applications, such as crowd motion, as well as cases involving discontinuous moving sets.


\begin{thebibliography}{10}
\providecommand{\url}[1]{{#1}}
\providecommand{\urlprefix}{URL }
\expandafter\ifx\csname urlstyle\endcsname\relax
  \providecommand{\doi}[1]{DOI~\discretionary{}{}{}#1}\else
  \providecommand{\doi}{DOI~\discretionary{}{}{}\begingroup
  \urlstyle{rm}\Url}\fi

\bibitem{MR2732595}
Acary, V., Bonnefon, O., Brogliato, B.: Nonsmooth modeling and simulation for
  switched circuits, \emph{Lect. Notes Electr. Eng.}, vol.~69.
\newblock Springer, Dordrecht (2011)

\bibitem{Acary-Brogliato-2008}
Acary, V., Brogliato, B.: Numerical Methods for Nonsmooth Dynamical Systems:
  Applications in Mechanics and Electronics.
\newblock Lect. Notes Appl. Comput. Mech. Springer Berlin Heidelberg (2008)

\bibitem{MR2115366}
Aussel, D., Daniilidis, A., Thibault, L.: Subsmooth sets: functional
  characterizations and related concepts.
\newblock Trans. Amer. Math. Soc. \textbf{357}(4), 1275--1301 (2005)

\bibitem{MR3719240}
Beck, A.: First-order methods in optimization, \emph{MOS-SIAM Ser. Optim.},
  vol.~25.
\newblock Society for Industrial and Applied Mathematics (SIAM), Philadelphia,
  PA; Mathematical Optimization Society, Philadelphia, PA (2017)

\bibitem{MR4621666}
Beck, A.: Introduction to nonlinear optimization---theory, algorithms, and
  applications with {P}ython and {MATLAB}, \emph{MOS-SIAM Ser. Optim.},
  vol.~32.
\newblock Society for Industrial and Applied Mathematics (SIAM), Philadelphia,
  PA; Mathematical Optimization Society, Philadelphia, PA (2023)

\bibitem{MR3025303}
Bounkhel, M.: Regularity concepts in nonsmooth analysis, \emph{Springer Optim.
  Appl.}, vol.~59.
\newblock Springer, New York (2012)

\bibitem{MR1870754}
Bounkhel, M., Thibault, L.: On various notions of regularity of sets in
  nonsmooth analysis.
\newblock Nonlinear Anal. \textbf{48}(2), 223--246 (2002)

\bibitem{MR2159846}
Bounkhel, M., Thibault, L.: Nonconvex sweeping process and prox-regularity in
  {H}ilbert space.
\newblock J. Nonlinear Convex Anal. \textbf{6}(2), 359--374 (2005)

\bibitem{MR3467591}
Brogliato, B.: Nonsmooth mechanics, third edn.
\newblock Commun. Numer. Methods Eng. Springer, [Cham] (2016).
\newblock Models, dynamics and control

\bibitem{MR2185604}
Brogliato, B., Daniilidis, A., Lemar\'{e}chal, C., Acary, V.: On the
  equivalence between complementarity systems, projected systems and
  differential inclusions.
\newblock Systems Control Lett. \textbf{55}(1), 45--51 (2006)

\bibitem{MR2731287}
Brogliato, B., Thibault, L.: Existence and uniqueness of solutions for
  non-autonomous complementarity dynamical systems.
\newblock J. Convex Anal. \textbf{17}(3-4), 961--990 (2010)

\bibitem{MR1058436}
Clarke, F.H.: Optimization and nonsmooth analysis, \emph{Classics Appl. Math.},
  vol.~5, second edn.
\newblock Society for Industrial and Applied Mathematics (SIAM), Philadelphia,
  PA (1990)

\bibitem{MR1488695}
Clarke, F.H., Ledyaev, Y.S., Stern, R.J., Wolenski, P.R.: Nonsmooth analysis
  and control theory, \emph{Grad. Texts in Math.}, vol. 178.
\newblock Springer-Verlag, New York (1998)

\bibitem{MR2768810}
Colombo, G., Thibault, L.: Prox-regular sets and applications.
\newblock In: Handbook of nonconvex analysis and applications, pp. 99--182.
  Int. Press, Somerville, MA (2010)

\bibitem{MR1189795}
Deimling, K.: Multivalued differential equations, \emph{De Gruyter Ser.
  Nonlinear Anal. Appl.}, vol.~1.
\newblock Walter de Gruyter \& Co., Berlin (1992)

\bibitem{MR110078}
Federer, H.: Curvature measures.
\newblock Trans. Amer. Math. Soc. \textbf{93}, 418--491 (1959)

\bibitem{MR4822735}
Garrido, J.G., Vilches, E.: Catching-{U}p {A}lgorithm with {A}pproximate
  {P}rojections for {M}oreau's {S}weeping {P}rocesses.
\newblock J. Optim. Theory Appl. \textbf{203}(2), 1160--1187 (2024)

\bibitem{MR2049810}
Goeleven, D., Brogliato, B.: Stability and instability matrices for linear
  evolution variational inequalities.
\newblock IEEE Trans. Automat. Control \textbf{49}(4), 521--534 (2004)

\bibitem{MR3574145}
Haddad, T., Noel, J., Thibault, L.: Perturbed sweeping process with a subsmooth
  set depending on the state.
\newblock Linear Nonlinear Anal. \textbf{2}(1), 155--174 (2016)

\bibitem{MR3626639}
Jourani, A., Vilches, E.: Moreau-{Y}osida regularization of state-dependent
  sweeping processes with nonregular sets.
\newblock J. Optim. Theory Appl. \textbf{173}(1), 91--116 (2017)

\bibitem{MR2404902}
Maury, B., Venel, J.: Un mod\`ele de mouvements de foule.
\newblock In: Paris-{S}ud {W}orking {G}roup on {M}odelling and {S}cientific
  {C}omputing 2006--2007, \emph{ESAIM Proc.}, vol.~18, pp. 143--152. EDP Sci.,
  Les Ulis (2007)

\bibitem{MR637727}
Moreau, J.J.: Rafle par un convexe variable. {I}.
\newblock In: Travaux du {S}\'{e}minaire d'{A}nalyse {C}onvexe, {V}ol. {I},
  Secr\'{e}tariat des Math\'{e}matiques, Publication, No. 118, pp. Exp. No. 15,
  43. Univ. Sci. Tech. Languedoc, Montpellier (1971)

\bibitem{MR637728}
Moreau, J.J.: Rafle par un convexe variable. {II}.
\newblock In: Travaux du {S}\'{e}minaire d'{A}nalyse {C}onvexe, {V}ol. {II},
  Secr\'{e}tariat des Math\'{e}matiques, Publication, No. 122, pp. Exp. No. 3,
  36. Univ. Sci. Tech. Languedoc, Montpellier (1972)

\bibitem{Moreau_2004}
Moreau, J.J.: An introduction to Unilateral Dynamics, pp. 1--46.
\newblock Springer Berlin Heidelberg (2004).
\newblock \doi{10.1007/978-3-540-45287-4_1}

\bibitem{Moreau_2011}
Moreau, J.J.: On Unilateral Constraints, Friction and Plasticity, p. 171–322.
\newblock Springer Berlin Heidelberg (2011).
\newblock \doi{10.1007/978-3-642-10960-7_7}

\bibitem{MR2142598}
Nesterov, Y.: Introductory lectures on convex optimization, \emph{Applied
  Optimization}, vol.~87.
\newblock Kluwer Academic Publishers, Boston, MA (2004).
\newblock A basic course

\bibitem{MR3286703}
Noel, J., Thibault, L.: Nonconvex sweeping process with a moving set depending
  on the state.
\newblock Vietnam J. Math. \textbf{42}(4), 595--612 (2014)

\bibitem{MR2527754}
Papageorgiou, N.S., Kyritsi-Yiallourou, S.T.: Handbook of applied analysis,
  \emph{Adv. Mech. Math.}, vol.~19.
\newblock Springer, New York (2009)

\bibitem{MR1694378}
Poliquin, R.A., Rockafellar, R.T., Thibault, L.: Local differentiability of
  distance functions.
\newblock Trans. Amer. Math. Soc. \textbf{352}(11), 5231--5249 (2000)

\bibitem{MR4659163}
Thibault, L.: Unilateral variational analysis in {B}anach spaces. {P}art
  {II}---special classes of functions and sets.
\newblock World Scientific Publishing Co. Pte. Ltd., Hackensack, NJ (2023)

\bibitem{MR3956966}
Vilches, E.: Existence and {L}yapunov pairs for the perturbed sweeping process
  governed by a fixed set.
\newblock Set-Valued Var. Anal. \textbf{27}(2), 569--583 (2019)

\end{thebibliography}
\end{document}